\documentclass[12pt]{amsart}

\usepackage{amsmath}
\usepackage{amsthm}
\usepackage{amsbsy}
\usepackage{amssymb}
\usepackage{mathtools}
\usepackage{calrsfs}
\usepackage{comment}
\usepackage[shortlabels]{enumitem}
\usepackage[normalem]{ulem}
\usepackage{amsrefs}
\usepackage{tikz}
\usetikzlibrary{positioning,patterns,calc}
\usepackage{mathrsfs}
\DeclareMathAlphabet{\mathpzc}{OT1}{pzc}{m}{it}
\usepackage[utf8]{inputenc}
\usepackage{hyperref}
\usepackage{comment}

\newtheorem{theorem}{Theorem}[section]
\newtheorem{lemma}[theorem]{Lemma}
\newtheorem{remark}[theorem]{Remark}
\newtheorem{proposition}[theorem]{Proposition}
\newtheorem{corollary}[theorem]{Corollary}

\newtheorem{definition}[theorem]{Definition}
\newtheorem*{theorem*}{Theorem}

\makeatletter
\@addtoreset{claim}{theorem}
\makeatother

\theoremstyle{plain}

\theoremstyle{plain}

\usepackage{enumitem}
\usepackage{lipsum}
\setlist[itemize]{leftmargin=*}

\def\bpf{\begin{proof}}
	\def\epf{\end{proof}}
\def\be{\begin{equation}}
	\def\ee{\end{equation}}
\def\bea{\begin{eqnarray}}
	\def\eea{\end{eqnarray}}
\def\bt{\begin{theorem}}
	\def\et{\end{theorem}}
\def\bl{\begin{lemma}}
	\def\el{\end{lemma}}
\def\br{\begin{remark}}
	\def\er{\end{remark}}
\def\bc{\begin{corollary}}
	\def\ec{\end{corollary}}
\def\bd{\begin{definition}}
	\def\ed{\end{definition}}
\def\bp{\begin{proposition}}
	\def\ep{\end{proposition}}

\def\h{\mathbb{H}^{n+1}}
\def\p{\partial}
\def\nun{\nu_{n+1}}
\def\zn{z_{n+1}}

\usepackage{etoolbox}
\makeatletter
\patchcmd{\@maketitle}
{\ifx\@empty\@dedicatory}
{\ifx\@empty\@date \else {\vskip3ex \centering\footnotesize\@date\par\vskip1ex}\fi
	\ifx\@empty\@dedicatory}
{}{}
\patchcmd{\@adminfootnotes}
{\ifx\@empty\@date\else \@footnotetext{\@setdate}\fi}
{}{}{}
\makeatother

\title[Asymptotic Plateau problem ]{Asymptotic Plateau problem via equidistant hyperplanes}

\author{Han Hong}
\address{Yau Mathematics Science Centre, Tsinghua University, 100084, Beijing, China}
\email{hh0927@tsinghua.edu.cn}

\author{Haizhong Li}
\address{Department of Mathematical Sciences, Tsinghua University, 100084, Beijing, China}
\email{lihz@tsinghua.edu.cn}

\author{Meng Zhang}
\address{Department of Mathematical Sciences, Tsinghua University, 100084, Beijing, China}
\email{zhangmen20@mails.tsinghua.edu.cn}

\begin{document}
	\maketitle
	
	\begin{abstract}
		We show the existence of a complete, strictly locally convex hypersurface within $\h$ that adheres to a curvature equation applicable to a broad range of curvature functions. This hypersurface possesses a prescribed asymptotic boundary at infinity and takes the form of a geodesic graph over a smooth bounded domain $\Omega$ at infinity. It is approximated by the shape of geodesic graphs whose boundaries rest upon equidistant hyperplanes. Through this procedure, we establish an alternative method for constructing solutions to the asymptotic Plateau problem. The resulting solutions may differ from the classical ones, particularly in cases where uniqueness cannot be assured.
	\end{abstract}
	
	\section{Introduction}
	In this paper we investigate a prescribing curvature problem with boundary sitting on the equidistant hyperplane in hyperbolic space, and apply this result to obtain 
	the existence of the hypersurface of constant curvature with prescribed asymptotic boundary at infinity.
	
	Let 
	\[\mathbb{H}^{n+1}=\{(z,z_{n+1})\in\mathbb{R}^{n+1}:\ z_{n+1}>0,z\in\mathbb{R}^n\}\]
	with metric
	\[g^{\mathbb{H}^{n+1}}=\frac{1}{z_{n+1}^2}g^{\mathbb{R}^{n+1}}\]
	be the $n+1$ dimensional hyperbolic space with the half space model. Let $\partial_\infty \mathbb{H}^{n+1}$ denote the ideal boundary of $\mathbb{H}^{n+1}$ at infinity, which can be identified with $\mathbb{R}^n\times {0}$ in the half space model. We assume that $\Gamma$ is a compact $(n-1)$-dimensional embedded hypersurface in $\partial_\infty\mathbb{H}^{n+1}$, which can be viewed as the boundary of a smooth domain in $\mathbb{R}^n$. The compactness of $\Gamma$ is considered with respect to the topology in $\mathbb{R}^n$.
	
	The asymptotic Plateau problem aims to find a complete hypersurface $\Sigma$ in $\mathbb{H}^{n+1}$ satisfying the following conditions:
	\begin{equation}\label{equation1introduction}
		f(\kappa[\Sigma])=\sigma
	\end{equation}
	with
	\begin{equation}\label{boundaryconditionintroduction}
		\partial\Sigma=\Gamma
	\end{equation}
	Here, $\kappa[\Sigma]=(\kappa_1,\cdots,\kappa_n)$ denotes the hyperbolic principal curvatures of $\Sigma$, and $f$ is a smooth symmetric function of $n$ variables. The constant $\sigma$ is a given positive value.

	Asymptotic Plateau problem for area miniminzing variety was solved by Anderson \cite{anderson1982,anderson1983} by using  geometric measure theory while the regularity was studied by Hardt-Lin \cite{lin1987}. Later, Tonegawa \cite{tonegawa} used similar ideas and extended their results into CMC case as part of his Phd thesis. Lin \cite{lin1989} started to think of desired (minimal) objects as graphs over mean convex domains at infinity, then proved the existence by applying techniques from PDE. This idea has been applied very well in last few decades to get generalizations of previous results. For special curvature functions such as mean curvature and Gauss curvature,  in a series of works by Nelli-Spruck \cite{nellispruck1996}, Guan-Spruck \cite{guanspruck2000}, Labourie \cite{labourie1991} and Rosenberg-Spruck \cite{rosenberg1994} they have obtained the existence and regularity of the solution while results for general curvature functions defined in the positive once were obtained in \cite{guan2009jga,guansurvey,guanjems,xiaoling2014jdg} by Guan, Spruck, Szapiel and Xiao. See also recent works by Wang \cite{wang2022} and Lu \cite{lu2022} in which they find a solution for some particular curvature functions in a general cone. When the curvature function of the hypersurface is equal to a given non-costant function defined in $\h$, Sui \cite{sui2021} and together with Sun \cite{sui2022} proved the existence under the assumption that the sbsolution exists.
	
	Above mentioned results in previous paragraph are all based on upper half space model and share similar spirit, that is to first seek the existence of compact hypersurfaces of constant curvature with boundary sitting in a horizontal horosphere that has small Euclidean height to the infinity. This requirement is essential in the proof of gradient estiamtes and boundary curvature estimates. We also need this assumption in our context. Furthermore, in the Minkowski model Lopez \cite{lopez1999,lopez2001}
	studied the existence of compact hypersurfaces of constant mean curvature with boundary on three different types of umbillic hypersurfaces. In particular, his results do not require the horosphere or equidistant hyperplane to be close to the infinity in the Euclidean sense. 
	
	It is well-known that there exist four different types of complete umbilical hypersurfaces in hyperbolic space:
	(1) geodesic sphere ($\kappa_i>1$); 
	(2) horosphere ($\kappa_i=1$);
	(3) equidistant hyperplane ($\kappa_i<1$);
	(4) geodesic hyperplane ($\kappa_i=0$).
	We will concentrate on equidistant hyperplanes for seeking the existence of compact hypersurfaces of constant curvature throughout this paper,  for convenience we choose, $0<\epsilon<\pi/2,$
	\begin{equation}\label{equidistant plane}
		P_\epsilon=\{(z_1,\cdots,z_{n+1})\in \mathbb{H}^{n+1}|z_{n}>0,\ \cos{\epsilon}z_{n+1}=\sin{\epsilon}z_n\}.
	\end{equation}
	This is a equidistant hyperplane with constant curvature equal to $\cos\epsilon$ (see Definition \ref{thetahyperplane}).
	
	In this paper, we shall focus on the existence of locally convex hypersurface, that is, the (hyperbolic) principal curvatures are positive everywhere. The class of interest, denoted by $\mathcal{C}_n$, consists of smooth symmetric functions defined on the closure of the positive cone  $K_n^+:=\{\lambda\in\mathbb{R}^{n}:\lambda_i>0, \ \forall i=1,\ldots,n\},$ which satisfy
	\begin{equation}\label{f=0 on the boundary}
		f=0\text{ on }\partial K_n^+,
	\end{equation}
	\begin{align}  \label{fi>0}
		f_i(\lambda)=\frac{\p f}{\p \lambda_i}>0, \ \ \forall 1\le i\le n,
	\end{align}
	\begin{equation}\label{f concave}  
		f\text{ is a concave function },
	\end{equation}
	\begin{equation}\label{f normalized}
		f\ \text{is normalized}: f(1,\cdots,1)=1,
	\end{equation}
	\begin{equation}\label{f degree 1}
		f \ \text{is homogeneous of degree one:} \ f(t\lambda)=tf(\lambda)\text{, }\forall\ t> 0.
	\end{equation}
	Also important is the subclass $\mathcal{S}_n$ consisting of functions in $\mathcal{C}_n$ satisfying a technical assumption:
	\begin{equation}\label{f unn condition}
		\lim_{R\to+\infty}f(\lambda_1,\cdots,\lambda_{n-1},\lambda_n+R)\ge1+\epsilon_0\text{  uniformly in }B_{\delta_0}(\bold{1}),
	\end{equation}
	where $\epsilon_0>0$ and $\delta_0>0$, $B_{\delta_0}(\bold{1})$ is the ball of radius $\delta_0$ centered at $\bold{1}=(1,\cdots,1)\in\mathbb{R}^n$. Note that the functions in $\mathcal{C}_n$ are positive in $K_n^+$ since $f(\lambda)=\sum\lambda_if_i>0$. The assumption (\ref{f unn condition}) is purely for deriving a normal second derivative estimate. This can be removed to find the existence using ideas from Xiao \cite{xiaoling2014jdg} which we will mention before Corollary \ref{using idea from xiao}.
	
	One can check the standard example $f=(\frac{H_n}{H_l})^{\frac{1}{n-l}}$, $0\le l<n$ is in $\mathcal{S}_n$, where $H_l$ is the normalized $l$-th elementary symmetric polynomial, and the proof of monontonicity and concavity can be found in Theorem 15.17 and 15.18 in \cite{lieberman}.
	Combining (\ref{f concave}), (\ref{f normalized}), (\ref{f degree 1}), following two inequalities hold.
	\begin{equation}\label{f upper bounder}
		f(\lambda)\le f(\bold{1})+\sum f_i(\bold{1})(\lambda_i-1)=\sum f_i(\bold{1})\lambda_i=\frac{1}{n}\sum \lambda_i\text{ in }K_n^+
	\end{equation}
	\begin{equation}\label{sum f_i lower bound}
		\sum f_i(\lambda)=f(\lambda)+\sum f_i(\lambda)(1-\lambda_i)\ge f(\bold{1})=1\text{ in }K_n^+.
	\end{equation}
	
	Compared with previous works, we deal with angle functions that naturally appear if we consider geodesic graph (see Definition \ref{cylindericalgraph}) instead of vertical graph. Suppose $\Sigma$ is a connected, orientable compact embedded hypersurface in $\h$.  For any point $q=(z_1,\cdots,z_n,z_{n+1})\in\Sigma$, let $\gamma(q)$ be the complete hyperbolic geodesic (Euclidean half circle) normal to $P_\epsilon$ with the intersection  point $q$. Then $\gamma(q)$ intersects with $\mathbb{R}^n\times\{0\}$ at $(y_1,\cdots,y_n)$ and let $u(y)=\arctan{\frac{z_{n+1}}{z_n}}$ be the Euclidean angle of $p$ and $q$. 
	
	Assume that $\Sigma$ is a geodesic graph over a bounded domain $\Omega\subset\mathbb{R}^n\times\{0\}$, the asymptotic Plateau problem $(\ref{equation1introduction})-(\ref{boundaryconditionintroduction})$ can be written as the following Dirichlet problem for a fully nonlinear second order equation $(u>0)$:
	\begin{align}\label{graph equation}
		\begin{cases}
			G(D^2u,Du,u,y)=\sigma &\text{in} \ \Omega,\\
			u=0 &\text{on}\ \p\Omega.
		\end{cases}
	\end{align}
	As we will see in Section \ref{parametrizationofgraph}, above equation is singular where $u=0$. We then approximate the boundary condition by $u=\epsilon$ on the boundary, and thus concentrate on the following new equation
	\begin{align}\label{newboundarycondition}
		\begin{cases}
			G(D^2u,Du,u,y)=\sigma &\text{in} \ \Omega,\\
			u=\epsilon &\text{on}\ \p\Omega.
		\end{cases}
	\end{align}
	
	We call solution to $(\ref{graph equation})$ or $(\ref{newboundarycondition})$ admissible if $graph(u)$ is strictly convex, i.e., hyperbolic principal curvatures are strictly positive. Our theorem can be described as 
	\begin{theorem}\label{thm with bdy}
		Let $\Omega$ be a bounded smooth domain in $\mathbb{R}^n$ and $f\in \mathcal{S}_n$. Then for any $\epsilon>0$ sufficiently small, there exists an admissible solution $u^\epsilon\in C^{\infty}(\bar{\Omega})$ of (\ref{newboundarycondition}) for any $\sigma\in(0,\cos\epsilon)$. Moreover, $u^\epsilon$ satisfies  the a priori estimates
		\begin{equation}
			|Du^{\epsilon}|\le C\text{ in }\Omega
		\end{equation}
		and 
		\begin{equation}
			\sin{u^\epsilon}\cdot|D^2u^\epsilon|\le C\text{ in }\Omega
		\end{equation}
		where $C$ is independent of $\epsilon$.
	\end{theorem}
	
	Based on this theorem and the standard procedures, we are able to give a solution to the asymptotic Plateau problem.
	
	\begin{theorem}\label{maintheorem2}
		Let $\Omega$ be a bounded smooth domain in $\mathbb{R}^n$ with boundary $\Gamma$ and $\sigma\in(0,1)$. Suppose that $f\in \mathcal{S}_n$. Then there exists a complete locally strictly convex hypersurface $\Sigma\subset \h$ satisfying (\ref{equation1introduction})-(\ref{boundaryconditionintroduction}) with uniformly bounded principal curvatures
		\[|\kappa[\Sigma]|\leq C \ \ \text{on}\ \ \Sigma.\]
		Moreover, $\Sigma$ is the graph of an admissible solution $u\in C^\infty(u)\cap C^1(\bar{\Omega})$ of the Dirichlet problem (\ref{graph equation}). Furthermore, $u^2\in C^\infty(\Omega)\cap C^{1,1}(\bar{\Omega})$ and
		\[\sqrt{1+z_n^2|Du|^2}=\frac{1}{\sigma} \ \text{on}\ \p\Omega,\ \ \ \ \ \ \ \sqrt{1+z_n^2|Du|^2}\leq \frac{1}{\sigma} \ \text{in}\ \Omega,\]
		\[\sin u\cdot |D^2u|\leq C\ \ \ \text{in}\ \ \Omega.\]
	\end{theorem}

	The curvature estimate $(1)$ in Theorem \ref{interior C2 esitimates thm} and the boundary curvature estimate in Theorem \ref{boundaryestiamtesofsecondderivatives} demonstrate that the hyperbolic principal curvatures of the admissible solutions of equation (\ref{newboundarycondition}) are uniformly bounded from above, regardless of $\epsilon$. Additionally, due to condition (\ref{f=0 on the boundary}) of the class $\mathcal{C}_n$, the hyperbolic principal curvatures also have a uniform positive lower bound, independent of $\epsilon$. Lemma \ref{c0estimate} provides a lower bound for admissible solutions on compact subsets of $\Omega$, ensuring that equation (\ref{graph equation}) is uniformly elliptic on such subsets. By applying the classical interior estimates of Evans and Krylov, we conclude that the admissible solutions possess uniform $C^{2,\alpha}$ estimates for any compact subdomain of $\Omega$. Ultimately, employing a diagonal sequence and standard arguments, we can establish the existence of a solution in Theorem \ref{maintheorem2} for any $0<\sigma<1$. It is important to note that the diagonal process also needs to be performed for $\sigma$, as we only have existence for $0<\sigma<\cos\epsilon$ in the compact case.
	
	The technical assumption (\ref{f unn condition}) is added for proving the boundary estimates for the second derivatives. We can borrow idea from  Xiao\cite{xiaoling2014jdg} to abandon this assumption. In fact, the global interior estimates also hold for hypersurfaces with asymptotic boundary as shown in part (2) of Theorem \ref{interior C2 esitimates thm}. For general curvature function $f\in \mathcal{C}_n$,  let $f^\theta:=\theta g+(1-\theta)f$ where $g\in\mathcal{S}_n$, then $f^\theta\in \mathcal{S}_n$. Thus we can use Theorem \ref{maintheorem2} to get a complete locally strictly convex hypersurface $\Sigma^\theta$ as a graph of $u^\theta$. Moreover, derivatives of $u^\theta$ have bounds independent of $\theta$. Thus letting $\theta$ tend to zero proves Theorem \ref{maintheorem2} without assumption (\ref{f unn condition}). Moreover, the domain $\Omega$ does not need to smooth by a standard approximation argument. 
	
	\begin{corollary}\label{using idea from xiao}
		Let $\Omega$ be a bounded $C^2$ domain in $\mathbb{R}^n$ with boundary $\Gamma$ and $\sigma\in(0,1)$. Suppose that $f\in \mathcal{C}_n$. Then there exists a complete locally strictly convex hypersurface $\Sigma\subset \h$ satisfying (\ref{equation1introduction})-(\ref{boundaryconditionintroduction}) with uniformly bounded principal curvatures
		\[|\kappa|\leq C \ \ \text{on}\ \ \Sigma.\]
		Moreover, $\Sigma$ is the graph of an admissible solution $u\in C^\infty(u)\cap C^1(\bar{\Omega})$ of the Dirichlet problem (\ref{graph equation}). Furthermore, $u^2\in C^\infty(\Omega)\cap C^{1,1}(\bar{\Omega})$ and
		\[\sqrt{1+z_n^2|Du|^2}=\frac{1}{\sigma} \ \text{on}\ \p\Omega,\ \ \ \ \ \ \ \sqrt{1+z_n^2|Du|^2}\leq \frac{1}{\sigma} \ \text{in}\ \Omega,\]
		\[\sin u\cdot |D^2u|\leq C\ \ \ \text{in}\ \ \Omega.\]
	\end{corollary}
	
	Uniqueness results were obtained in \cite{guansurvey,xiaoling2014jdg} if
	one of followings holds: 
	\begin{itemize}
		\item either $\Gamma$ is $C^{2,\alpha}$ and (Euclidean) mean convex, 
		\item or $\Gamma$ is $C^2$ and star-shaped about the origin, 
		\item or $\sum f_i>\sum \kappa_i^2 f_i $ in $K_n^+\cap \{0<f<1\}$ is satisfied.
	\end{itemize} Thus under these assumptions, our solutions in Theorem \ref{maintheorem2} and Corollary \ref{using idea from xiao} are exactly same as the ones obtained in \cite{guansurvey,xiaoling2014jdg}. In other cases, we possibly construct different solutions to asymptotic Plateau problem. It is an interesting question to determine whether our procedure provides new solutions in general cases.
	
	\vspace{0.5cm}
	
	The organization of the paper is as follows. In Section \ref{parametrizationofgraph}, we calculate geometric quantities for a geodesic graph and derive certain identities that will be used later. Section \ref{gradientestimatessection} is dedicated to establishing basic identities on a hypersurface in $\mathbb{H}^{n+1}$, which will serve as the foundation for deriving gradient estimates. Additionally, we provide a proof of $C^0$ estimates for the geodesic graph with a boundary located on equidistant hyperplanes. To prove the boundary estimates for the second derivatives, we employ a maximum principle for the partial linearized operator around boundary points. This analysis is presented in Section \ref{section4 boundary estimates}. In Section \ref{section5}, we focus on proving two estimates for the maximum hyperbolic principal curvature. One estimate pertains to compact hypersurfaces, while the other relates to global interior estimates for complete noncompact hypersurfaces. We demonstrate that their proofs can be unified through Theorem \ref{interior C2 esitimates thm}. Lastly, in the final section, we utilize the degree theory developed by Y.Y. Li to obtain a solution for the approximation equation. We note that the linearized operator may not always be invertible over the time interval $[0,1]$, rendering the method of continuity insufficient for the proof.
	
	\subsection{Acknowledgements}
	The first author would like to thank all the participants in the random hyperbolic surface seminar for discussion about the Anderson's paper.  The first author is supported by Shuimu Tsinghua Scholar Program, China Postdoctoral Science Foundation No.2021TQ0186 and Postdoctoral Exchange Fellowship Program No.YJ20210267. The second and third authors are partially supoorted by NSFC No. 11831005.
	
	\vspace{1cm}
	\section{Preliminaries}\label{parametrizationofgraph}

	Consider the space 
	$\mathbb{R}_+^{n+1}=\{(z_1,\cdots,z_n,z_{n+1}): z_{n+1}>0\}$ with the hyperbolic metric 
	\[g=\frac{dz_1^2+\cdots+dz_{n+1}^2}{z_{n+1}^2}.\]
	
	Suppose $\Sigma$ can be locally represented as the geodesic graph of a smooth function $u>0$ in a domain $\Omega\subset \mathbb{R}^{n}_+$ where $ \mathbb{R}^{n}_+=\{y=(y_1,\cdots,y_n): y_n>0\}.$ That is 
	\[\Sigma=\{(y_1,\cdots,y_{n-1},y_n\cos u,y_n\sin u)\in \mathbb{R}_+^{n+1}:\ y\in \Omega\}.\]
	As mentioned in the introduction,  $u$ is the angle function that represents the angle between the point on the graph and the domain. Denote the canonical basis of $\mathbb{R}^{n+1}$ by $E_1,\cdots,E_{n+1}.$
	
	In this scenario, the map $F(y_1,\cdots, y_n)=(y_1,\cdots,y_{n-1},y_n\cos u,y_n\sin u)$ sends $\p_{y_i}$ and $\partial_{y_n}$ to 
	\[F_i=E_i-(y_nu_i\sin u ) E_n+(y_nu_i\cos u )E_{n+1}, \ \ \ i=1,\ldots, n-1.\]
	and
	\[F_n=(\cos u-y_nu_n\sin u)E_n+(\sin u+y_nu_n\cos u)E_{n+1}.\]
	Thus the (Euclidean) unit normal vector field to $\Sigma$ is 
	\[\nu=\frac{(-y_nu_1,\cdots,-y_{n}u_{n-1},-\sin u-y_nu_n\cos u,\cos u-y_nu_n\sin u)}{\omega}\]
	where 
	\[\omega=\sqrt{1+y_n^2|\nabla u|^2}.\]
	
	The induced Euclidean metric and second fundamental form of $\Sigma$ are given by
	\[g^E_{ij}=\delta_{ij}+y^2_nu_iu_j\]
	and
	\[h_{ij}^E=\frac{\delta_{nj}u_i+\delta_{ni}u_j+y_n^2u_iu_ju_n+y_nu_{ij}}
	{\omega}.\]
	
	According to \cite{caffarelli1986}, the Euclidean principal curvature are the eigenvalues of the symmetric matrix $A^E[u]=[a_{ij}^E]$. Here
	\[a^E_{ij}=\gamma^{ik}h_{ij}^{E}\gamma^{lj}\]
	where $\gamma^{ij}$ is the inverse matrix of $\gamma_{ij}$ which is the square root of the Riemmanian metric $g^E_{ij}$, i.e., $\gamma_{ij}\gamma_{jk}=g^E_{ik}$. Then it is not hard to see that
	\[\gamma_{ij}=\delta_{ij}+\frac{y_n^2u_iu_j}{1+\omega}\]
	and
	\[\gamma^{ij}=\delta_{ij}-\frac{z_n^2u_iu_i}{\omega(1+\omega)}.\]
	
	From the relation between hyperbolic and Euclidean principal curvature (see equation (\ref{relationbetweenprincipalcurvatures})), we get that the hyperbolic principal curvatures of the hypersurface $\Sigma$ are the eigenvalues of the matrix $A[u]=[a_{ij}]$. Here,
	\[a_{ij}=\frac{y_nz_{n+1}}{\omega}\gamma^{ik}u_{kl}\gamma^{\ell j}+\frac{y^2_nz_{n+1}}{\omega^3}u_iu_ju_n+\frac{z_{n+1}}{\omega^2}(\gamma^{in}u_j+\gamma^{jn}u_i)+\nu_{n+1}\delta_{ij}.\]
	In the calculation we have used that $\gamma^{ij}u_j=\frac{u_i}{\omega}$. Note that $\nu_{n+1}=\frac{\cos u-y_nu_n\sin u}{\omega}$ and
	$z_{n+1}=y_n\sin u.$
	
	We define a vector field in $\h$
	\begin{equation}\label{wdefinition}W((z_1,\cdots,z_n,z_{n+1}))=(0,\cdots,0,\sin \theta,-\cos \theta),\end{equation}
	where $\cot\theta=z_n/z_{n+1}$. This vector is perpendicular to the equidistant hyperplane that passes through the origin and forms the angle $ \theta$ with the infinity (see definition below). For the convenience of the discussion, we define
	
	\begin{definition}\label{thetahyperplane}
		Throughout this paper, a (equidistant) hyperplane in $\mathbb{R}^{n+1}_+$ is called $\theta$-hyperplane, denoted by $P_\theta$,
		if it passes through the origin and forms an angle $\theta$ with $\mathbb{R}^{n}_+$.
	\end{definition}
	
	It is known that $\theta$-hyperplane is umbillic and has constant principal curvature $\cos\theta$ in $\h$.
	
	\begin{definition}\label{cylindericalgraph}
		We say that $\Sigma$ is a geodesic graph over a domain $\Omega\subset \mathbb{R}^n_+$ if 
		\[\nu\cdot W<0, \ \ \forall \ x\in\Omega.\]
	\end{definition}
	
	\vskip.3cm
	
	Let $\mathcal{M}$ be the vector space of $n \times n$ symmetric matrices and
	$$
	\mathcal{M}_K=\{A \in \mathcal{S}: \lambda(A) \in K\},
	$$
	where $\lambda(A)=\left(\lambda_1, \ldots, \lambda_n\right)$ denotes the eigenvalues of $A$. Define a function $F$ by
	$$
	F(A)=f(\lambda(A)), \quad A \in \mathcal{M}_K .
	$$
	Throughout the paper we denote
	$$
	F^{i j}(A)=\frac{\partial F}{\partial a_{i j}}(A), \quad F^{i j, k l}(A)=\frac{\partial^2 F}{\partial a_{i j} \partial a_{k l}}(A) .
	$$
	The matrix $\left\{F^{i j}(A)\right\}$ is symmetric and its eigenvalues are $f_1, \ldots, f_n$, and therefore it is positive definite for $A \in \mathcal{M}_K$ if $f$ satisfies $(\ref{fi>0})$, while $(\ref{f concave})$ implies that $F$ is concave for $A \in \mathcal{S}_K$, that is
	$$
	F^{i j, k l}(A) \xi_{i j} \xi_{k l} \leq 0, \quad \forall\left\{\xi_{i j}\right\} \in \mathcal{M}, A \in \mathcal{M}_K .
	$$
	We have that
	$$
	\begin{aligned}
		F^{i j}(A) a_{i j} & =\sum f_i(\lambda(A)) \lambda_i, \\
		F^{i j}(A) a_{i k} a_{j k} & =\sum f_i(\lambda(A)) \lambda_i^2 .
	\end{aligned}
	$$

	Now, recall the function $G$ in $(\ref{graph equation})$ is determined by
	\begin{equation*}\label{definitionofG}
		G(D^2u,Du,u,y)=F(A[u]),
	\end{equation*}
	where $A[u]=[a_{ij}]$. We calculate that
	\[G^{st}:=\frac{\p G}{\p u_{st}}=\frac{y_nz_{n+1}}{\omega}F^{ij}\gamma^{is}\gamma^{tj}\]
	and
	\[G^{st}u_{st}=F^{ij}a_{ij}-\nu_{n+1}\sum F^{ii}-\frac{y^2_nz_{n+1}}{\omega^3}F^{ij}u_iu_j-\frac{z_{n+1}}{\omega^2}F^{ij}(\gamma^{in}u_j+\gamma^{jn}u_i).\]
	Moreover,
	\begin{align*}
		G_u:=\frac{\p G}{\p u}&=F^{ij}(\frac{\p\nu_{n+1}}{\p u}\delta_{ij}+z_n\cos u\ a_{ij}^E)\\
		&=(\nu_{n}-\nu_{n+1}\cot u)\sum F^{ii}+F^{ij}a_{ij}\cot u.
	\end{align*}
	The calculation of $G^s=\partial G/\p u_s$ is more complicated so we compute it later when it is needed. It follows from above that the condition $(\ref{fi>0})$ implies the equation $(\ref{graph equation})$ is elliptic for $u$ if $A[u]\in \mathcal{M}_K$ while $(\ref{f concave})$ implies that $G(D^2u,Du,u,y)$ is concave with respect to $D^2u.$

	\section{Gradient estimates}\label{gradientestimatessection}
	In this section, we obtain some gradient estimates for the fully nonlinear elliptic equation.  
	
	Since the hypersurface $\Sigma$ in $\h$ is also a hypersurface in $\mathbb{R}^{n+1}$, we add $E$  to distinguish quantities in two cases. For example, we denote the induced hyperbolic metric and Levi-Civita connection on $\Sigma$ by $g$ and $\nabla$, respectively while we denote induced Euclidean ones by $g^E$ and $\nabla^E$, respectively. We denote ambient connections in $\mathbb{R}^{n+1}$ and $\h$ by $D^E$ and $D$, respectively. All the calculations in this section are local, so we choose local coordinates $\{x_1,\cdots,x_n\}$ and  the local frame $\{e_i\}_{i=1}^n$ on $\Sigma$ where $e_i=\p_{x_i}$. Let $\nu$ be the unit normal vector to $\Sigma$ with respect to the Euclidean metric, this also determines a unit normal vector $\bold{n}$ with respect to the hyperbolic metric by the relation
	\[\bold{n}=z_{n+1}\nu.\]
	Denote $\nu_{n+1}=\nu\cdot E_{n+1}$ and $\nu_n=\nu\cdot E_n$ where $\cdot$ is the usual dot product in $\mathbb{R}^{n+1}.$ Sometimes we also use $<,>$ to represent $\cdot.$
	
	The hyperbolic metric $g$ and Euclidean metric $g^E$ of $\Sigma$ satisfies
	\[g^E_{ij}=z_{n+1}^2 g_{ij}.\]
	The hyperbolic second fundamental form and Euclidean second fundamental form are
	\[h_{ij}=g(D_{e_i}e_j,\bold{n}),\ \ \text{and}\ \ \ \ \ h^E_{ij}=g^E(D^E_{e_i}e_j,\nu)\]
	satisfying
	\begin{equation}\label{relationoftwofundamentalforms}
		h_{ij}=\frac{1}{z_{n+1}}h^E_{ij}+\frac{\nu_{n+1}}{z_{n+1}^2}g^E_{ij}.\end{equation}
	Thus
	\begin{equation}\label{relationbetweenprincipalcurvatures}
		\kappa_i=z_{n+1}k_i^E+\nu_{n+1}
	\end{equation}
	where $\kappa^E_1,\cdots,\kappa^E_n$ and $\kappa_1,\cdots,\kappa_n$ are the hyperbolic and Euclidean principal curvatures, respectively. 
	
	In particular, hereafter we choose $e_1,\cdots, e_n$ to be a geodesic orthonormal frame with respect to the hyperbolic metric, then $g_{ij}=\delta_{ij}$ and $g^E_{ij}=z_{n+1}^2\delta_{ij}.$ For a smooth function $u$ on $\Sigma$, by standard calculation under conformal change we obtain that
	\begin{align}\label{hessianchange}
		\nabla_{ij}u=\nabla^E_{ij}u+\frac{1}{z_{n+1}}[(z_{n+1})_iu_j+(z_{n+1})_ju_i-(z_{n+1})_ku_k\delta_{ij}].
	\end{align}
	Here $\nabla_{ij}u=\nabla^2u(e_i,e_j).$
	
	\subsection{Test functions}The following results will be used in the proof of gradient estimates.
	\begin{proposition}\label{hessianofnu}
		Under above setting, we have
		that for $i,j=1,\cdots,n$
		\begin{align*}
			\nabla^E_{ij}\nu_{n}&=-\frac{h^E_{ik}h^E_{kj}\nu_n+(\nabla^E_kh^E_{ij})(z_n)_k}{z_{n+1}^2}+\frac{2h^E_{ij}(\zn)_k(z_n)_k}{z^3_{n+1}}\\
			&+\frac{h^E_{kj}((\zn)_i(z_n)_k-(\zn)_k(z_n)_i)}{z^3_{n+1}}+\frac{2h^E_{ik}((\zn)_j(z_n)_k-(z_n)_j(z_{n+1})_k)}{z^3_{n+1}}
		\end{align*}
		and
		
		\begin{align*}
			\nabla^E_{ij}\nu_{n+1}&=-\frac{h^E_{ik}h^E_{kj}\nu_{n+1}+(\nabla^E_kh^E_{ij})(z_{n+1})_k}{z_{n+1}^2}+\frac{2h^E_{ij}(\zn)_k^2}{z^3_{n+1}}.
		\end{align*}

	\end{proposition}
	
	\begin{proof}
		First, we have that for $\ell=1,\cdots,n+1,$
		\begin{equation}\label{p3.1.1}
			\nabla^E_{\p_i}\nu_{l}=-\frac{h_{ik}^E(z_{l})_k}{z_{n+1}^2}\end{equation}
		and
		\begin{equation}\label{p3.1.2}
			\nabla^E_{\p_i}\p_j=\frac{(z_{n+1})_i\p_j+(z_{n+1})_j\p_i-\delta_{ij}(z_{n+1})_m\p_m}{z_{n+1}},
		\end{equation}
		in particular, $\nabla^E_{\partial_i}\p_j=\nabla^E_{\partial_j}\p_i.$
		
		Moreover,
		\begin{align}\label{p3.1.3}
			\nabla^E _{kj} z_\ell&=h^E_{jk}\nu_l+\frac{<\nabla^E_{\partial_j}\partial_k,\p_m >(z_l)_m}{\zn^2}  \nonumber\\
			&=h^E_{jk}\nu_l+\frac{(z_{n+1})_j(z_l)_k+(z_{n+1})_k (z_l)_j-\delta_{jk}(z_{n+1})_m(z_l)_m}{z_{n+1}}
		\end{align}
		By Codazzi equation, we have
		\begin{equation}\label{p3.1.4}
			h^E_{ikj}=h^E_{ijk}-h^E(\nabla^E_{\p_k}\p_i,\p_j)+h^E(\nabla^E_{\p_j}\p_i,\p_k)
		\end{equation}
		Then it follows from (\ref{p3.1.1}),(\ref{p3.1.2}),(\ref{p3.1.3}) and (\ref{p3.1.4}) that
		\begin{align*}
			\nabla^E_{ij}\nu_l&=(\nu_l)_{ij}-(\nabla^E_{\partial_i}\p_j)\nu_l\\
			&=-\frac{h^E_{ikj}(z_{l})_k+h^E_{ik}(z_{l})_{kj}}{z^2_{n+1}}+2\frac{h^E_{ik}(z_{l})_k(\zn)_j}{\zn^3}+\frac{h^E(\nabla^E_{\p_i}\p_j,\p_k)}{z^2_{n+1}}    (z_l)_k  \\
			&=-\frac{h^E_{ik}h^E_{kj}\nu_l+(\nabla^E_kh^E_{ij})(z_l)_k}{z_{n+1}^2}+\frac{2h^E_{ik}((\zn)_j(z_l)_k-(\zn)_j(z_{n+1})_k)}{z^3_{n+1}}\\
			&+\frac{h^E_{ij}(z_{n+1})_m^2}{z^3_{n+1}}+\frac{h^E(\nabla^E_{\p_k}\p_i,\p_j)}{z^2_{n+1}}(z_l)_k\\
			&=-\frac{h^E_{ik}h^E_{kj}\nu_l+(\nabla^E_kh^E_{ij})(z_l)_k}{z_{n+1}^2}+\frac{2h^E_{ij}(\zn)_k(z_l)_k}{z^3_{n+1}}\\
			&+\frac{h^E_{kj}((\zn)_i(z_l)_k-(\zn)_k(z_l)_i)}{z^3_{n+1}}+\frac{h^E_{ik}((\zn)_j(z_l)_k-(z_l)_j(z_{n+1})_k)}{z^3_{n+1}}.
		\end{align*}
		This yields identities in the Proposition for $
		\ell=n$ and $\ell=n+1$.
	\end{proof}

	\begin{theorem}\label{gradientcalculation}
		Let $\Sigma$ be a smooth hypersurface in $\h$ satisfying $f(\kappa)=\sigma$ and $f\in \mathcal{C}_n$. Then in any local frame, we have
		\begin{equation}\label{t3.2.1}
			F^{ij}\nabla_{ij}\cot u=(\nu_n-\nu_{n+1}\cot u )\sigma+\cot u\sum f_i
		\end{equation}
		and 
		\begin{equation}\label{t3.2.2}
			F^{ij}\nabla_{ij}(\nu_n-\nu_{n+1}\cot u)=-\sigma \cot u -(\nu_n-\nun\cot u )\sum f_i\kappa_i^2.
		\end{equation}
	\end{theorem}
	
	\begin{proof}
		Since $\cot u=\frac{z_n}{z_{n+1}}$, we have that
		\[\partial_j\cot u=\frac{(z_n)_jz_{n+1}-(z_{n+1})_jz_n}{z_{n+1}^2}\]
		thus
		\begin{align*}
			\nabla^E_{ij}\cot u&=\frac{(z_n)_j(z_{n+1})_i-(z_n)_i(z_{n+1})_j+z_{n+1}h^E_{ij}\nu_n-z_nh_{ij}^E\nu_{n+1}}{z_{n+1}^2}\\
			&+\frac{-2z_{n+1}^2(z_{n+1})_i(z_n)_j+2z_nz_{n+1}(z_{n+1})_i(z_{n+1})_j}{z_{n+1}^4}.
		\end{align*}
		Furthermore, by (\ref{hessianchange}) we get that
		\begin{align*}
			\nabla_{ij}\cot u&=\frac{z_{n+1}\nu_n-z_n\nu_{n+1}}{z_{n+1}^2}h_{ij}^E-\frac{z_{n+1}(z_n)_k(z_{n+1})_k-z_n(z_{n+1})_k(z_{n+1})_k}{z^3_{n+1}}\delta_{ij}.
		\end{align*}
		Since $h^E_{ij}=z_{n+1}h_{ij}-z_{n+1}\nu_{n+1}\delta_{ij}$ and
		\[(z_n)_k(z_{n+1})_k=-z_{n+1}^2\nu_n\nu_{n+1},\]
		\[(z_{n+1})_k(z_{n+1})_k=z_{n+1}^2(1-\nu_{n+1}^2).\]
		Hence
		\[\nabla_{ij}\cot u=(\nu_n-\nu_{n+1}\cot u)h_{ij}+\cot u \delta_{ij}.\]
		This proves the first identity (\ref{t3.2.1}).
		
		We now continue to prove the second one. We calculate $\nabla_{ij}\nu_{n}$ and $\nabla_{ij}\nun$ at first.

		It follows from (\ref{hessianchange}) that
		\[\nabla_{ij}\nu_n=\nabla_{ij}^E\nu_n+\frac{1}{z_{n+1}}[(z_{n+1})_i(\nu_n)_j+(z_{n+1})_j(\nu_n)_i-(z_{n+1})_k(\nu_n)_k\delta_{ij}].\]
		Then
		\begin{align}\label{fijnun+1}
			F^{ij}\nabla_{ij}\nu_{n}&=-\frac{\nu_n}{z_{n+1}^2} F^{ij}h^E_{ik}h^E_{kj}-\frac{(z_{n})_k}{z_{n+1}^2}F^{ij}\nabla^E_kh^E_{ij}+\frac{2}{z_{n+1}}F^{ij}(z_{n+1})_i(\nu_n)_j\\
			&+\frac{2F^{ij}h^E_{ij}(z_{n+1})_k(z_n)_k}{z^3_{n+1}}-\frac{1}{z_{n+1}}(z_{n+1})_k(\nu_n)_k\sum F^{ii}.\nonumber
		\end{align}
		By differentiating the equation
		\[\sigma=F(h_{ij})=F(\{h^E_{ij}/u+\nu_{n+1}\delta_{ij}\}),\]
		we obtain
		\[F^{ij}(z_{n+1}\nabla_k^Eh^E_{ij}-(z_{n+1})_kh^E_{ij}+(\nun)_{k}z_{n+1}^2\delta_{ij})=0.\]
		Thus, by (\ref{relationoftwofundamentalforms}) we have
		\[F^{ij}\nabla_k^Eh^E_{ij}=(z_{n+1})_k(\sigma-\nun\sum F^{ii})-(\nun)_kz_{n+1}\sum F^{ii}.\]
		Furthermore, 
		\begin{align*}
			F^{ij}h^E_{ik}h^E_{kj}&=z_{n+1}^2F^{ij}(h_{ik}-\nun \delta_{ik})(h_{kj}-\nun \delta_{kj})\\
			&=z^2_{n+1}(\sum f_i\kappa_i^2-2\sigma\nun+\nun^2\sum f_i).
		\end{align*}
		Hence, it follows from (\ref{fijnun+1}) and Proposition \ref{hessianofnu} that
		\begin{align}\label{t3.2.3}
			F^{ij}\nabla_{ij}\nu_{n}&=-\nu_n (\sum f_i\kappa_i^2-2\sigma\nun+\nun^2\sum f_i)+\frac{2}{z_{n+1}}F^{ij}(z_{n+1})_i(\nu_n)_j\\
			&+\frac{(z_n)_k(\nun)_k -(z_{n+1})_k(\nu_n)_k}{z_{n+1}}\sum F^{ii}+\frac{(z_{n+1})_k(z_n)_k}{z_{n+1}^2}(\sigma-\nun\sum F^{ii})\nonumber\\
			&+2\left(\frac{F^{ij}h^E_{ik}((z_{n+1})_j(z_n)_k-(z_{n})_j(z_{n+1})_k)}{z^3_{n+1}}\right)\nonumber\\
			&=-\nu_n (\sum f_i\kappa_i^2-2\sigma\nun+\nun^2\sum f_i)-\frac{2 F^{ij}h^E_{ik}(z_{n})_j(z_{n+1})_k}{z^3_{n+1}}\nonumber\\
			&+\frac{(z_{n+1})_k(z_n)_k}{z_{n+1}^2}(\sigma-\nun\sum F^{ii})\nonumber
		\end{align}
		where we have used (\ref{p3.1.1}) in the second equality.
		
		Similarly, we can calculate by following above lines to obtain
		\begin{align}\label{t3.2.4}
			F^{ij}\nabla_{ij}\nu_{n+1}&=-\nun (\sum f_i\kappa_i^2-2\sigma\nun+\nun^2\sum f_i)\\
			&+\frac{2}{z_{n+1}}F^{ij}(z_{n+1})_i(\nun)_j+\frac{(z_{n+1})^2_k}{z_{n+1}^2}(\sigma-\nun\sum F^{ii})\nonumber.
		\end{align}
		
		Finally, since
		\[\sum (z_{n+1})_k(z_n)_k=z^2_{n+1}(E_{n+1}-\nu_{n+1}\nu)(E_n-\nu_n\nu)=-z^2_{n+1}\nu_n\nu_{n+1},\]
		\[\sum (z_{n+1})_k(z_{n+1})_k=z^2_{n+1}(1-\nu_{n+1}^2),\]
		it follows from (\ref{t3.2.1}),(\ref{t3.2.3}),\ref{t3.2.4}) and (\ref{p3.1.1}) that
		\begin{align*}
			F^{ij}\nabla_{ij}(\nu_n-\nun \cot u)&= F^{ij}\nabla_{ij}\nu_n-\cot u F^{ij}\nabla_{ij} \nun-\nun f^{ij}\nabla_{ij}\cot u\\
			&-2F^{ij}(\nun)_i(\cot u)_j\\
			&=-\sigma \cot u -(\nu_n-\nun\cot u )\sum f_i\kappa_i^2.
		\end{align*}
		This gives the equation (\ref{t3.2.2}), thus completing the proof of the theorem.
		
	\end{proof}
	\subsection{$C^0$ estimate and gradient estimate}

	A hypersurface in hyperbolic space $\h$ with a constant curvature function has an equidistant hyperplane as a natural barrier. An equidistant hyperplane with a constant curvature $\sigma\in(0,1)$ can either be a Euclidean half-hyperplane that intersects infinity at an angle of $\theta=\arccos\sigma$, or a hypersphere with a special center and radius. In this context, we utilize the latter to establish several bounds for the function $u$, whose graph constitutes the hypersurface.
	
	Fix $\epsilon>0$ a small positive number. Let $\Sigma$ be an orientable, embedded hypersurface in $\h$ with $\partial \Sigma\subset P_{\epsilon}$ (see Definition \ref{thetahyperplane}) such that it separates $P_\epsilon^+=\{(z,\theta)\in \h:\theta>\epsilon\}$ into a bounded region and an unbounded region. Let $\Omega$ be the bounded domain in $\mathbb{R}_+^n\times \{0\}$ such that its rotation  by angle $\epsilon$ to $P_\epsilon$, denoted by $\Omega_\epsilon$, has boundary $\partial\Sigma.$ Assume that $\kappa[\Sigma]\in K_+$ and $f(\kappa[\Sigma])=\sigma\in (0,\cos\epsilon)$. The calculations are done with respect to the unit outer normal that points into the unbounded region. We remind readers that $\partial \Omega$ doesn't have to be connected. In this section, we only require $f$ to be in $\mathcal{C}_n$, then the linearized operator of the fully nonlinear equation is elliptic so that we can make use of the maximum principle.
	
	We let $B_{1}, B_2$ be  balls in $\mathbb{R}^{n+1}$ of radius $R$ centered at $a=(a',-\sigma R)$ and $b=(b', \sigma R)$ where $\sigma\in (0,\cos\epsilon)$. Then $\p B_1\cap \h$ and $\p B_2\cap \h$ are hypersurfaces of constant curvature $\sigma$ with respect to its outer unit normal and inner unit normal, respectively. We have the following result (c.f. \cite{guanjems}).
	\begin{lemma}\label{geometricmaximumprinciple}
		The followings hold:
		\begin{enumerate}
			\item[(1),] $\Sigma$ is contained in the wedge by $P_\epsilon$ and the vertical geodesic hyperplane through origin.
			\item[(2),] If $\p\Sigma\subset B_1$, then $\Sigma\subset B_1.$
			\item[(3),] If $B_1\cap P_\epsilon\subset \Omega_\epsilon $, then $B_1\cap \Sigma=\emptyset.$
			\item[(4),]  If $B_2\cap \Omega^\epsilon=\emptyset$, then $B_2\cap\Sigma=\emptyset.$
		\end{enumerate}
	\end{lemma}
	\begin{proof}
		For (1), If $\Sigma$ goes through below $P_\epsilon$, we can translate $P_\epsilon$ along positive $z_n$ direction far away, then translate it back until there is a touching point violating the maximum principle. Similarly, we can show it does not pass through the vertical geodesic plane through origin. In fact, it does not intersect the vertical geodesic plane through the most left point of $\Omega_\epsilon$. The proofs of $(2),(3),(4)$ follow exact same lines of \cite[Lemma 3.1]{guanjems} by expanding or shrinking balls with respect to $(a',0)$ or $(b',0).$ We can also prove $(4)$ by translating $B_2$ horizontally and using maximum principle. The key to prove all above statements is to avoid interior touching point with same directions of curvature.
	\end{proof}
	
	The conclusion (1) of Lemma \ref{geometricmaximumprinciple} implies   natural bounds: $\epsilon<u<\pi/2$. However, this is not enough, we need an upper bound depending on $\Omega$ and $\epsilon$ that is away from $\pi/2$, a lower bound depending on point $z$, $\epsilon$  and $\Omega$ that is uniform away from zero as $\epsilon\rightarrow 0$ in compact subsets of $\Omega$. 
	
	Let $d(\Omega)$ be the Euclidean diameter of $\Omega$ and let $D_{d(\Omega)/2}(p)$ be the n-dimensional ball in $P_\epsilon$ containing $\Omega_\epsilon$ centered at $p\in P_\epsilon$. Let $\ell(\Omega)$ be the distance between $p$ and $\partial \mathbb{R}^n_+$, i.e., $\ell(\Omega)=p\cdot z_n$. For a fixed point $z\in\Omega$, denote $d(z)=\operatorname{dist}(z,\p\Omega)$ and $\ell(z)=z\cdot E_n$. Let $D_{d(z)}(\tilde{z})$ be the n-dimensional ball in $P_\epsilon$ centered at $\tilde{z}$ (this is corresponding point of $z$ in $\Omega_\epsilon$) of radius $d(z)$ that is contained in $\Omega_\epsilon$.
	
	\begin{lemma}\label{c0estimate}

		There exist positive constants $C_1$ depending on $\epsilon, d(\Omega), \ell(\Omega)$ and $\sigma$ such that
		\[u\leq C_1<\frac{\pi}{2}\]
		and $C_2$ depending on $\epsilon,d(z),\ell(z)$ and $\sigma$ such that 
		\[u(z)\geq C_2>\epsilon\]
		in compact subsets of $\Omega.$

	\end{lemma}
	
	\begin{proof}
		Let $B_1$ be a ball in $\mathbb{R}^{n+1}$ centered at $a=(a',-\sigma R)$ with radius R such that $B_1\cap P_\epsilon=D_{d(\Omega)/2}(p)$. By (2) of Lemma \ref{geometricmaximumprinciple}, $\Sigma$ is completely contained in $B_1\cap P_\epsilon^+$. Then 
		\[u\leq  \bar{u}\]
		where $\bar{u}$ is the maximal angle of point on $\partial B_1\cap P_\epsilon^+$  satisfying
		\[\cos \bar{u}=\frac{\sigma R}{R-(\ell(\Omega)-\sigma R \tan \bar{u})\sin \bar{u}}\]
		equivalent to
		\[\ell(\Omega)\sin \bar{u}+\sigma R \cos \bar{u}=R.\]
		So
		\begin{equation}\label{lemma3.4.1}
			\cos \bar{u}=\frac{\sigma R^2+\ell\sqrt{\ell^2+\sigma^2R^2-R^2}}{\sigma^2R^2+\ell^2}<1.
		\end{equation}
		By simple Euclidean geometry, $R$ satisfies
		\[(\ell(\Omega)\tan\epsilon+\sigma R/\cos\epsilon )^2+(d(\Omega)/2)^2=R^2.\]
		This implies that
		\begin{equation}\label{lemma3.4.2}
			\frac{\sigma \ell\sin\epsilon}{\cos^2\epsilon-\sigma^2}+\frac{d\cos\epsilon}{2\sqrt{\cos^2\epsilon-\sigma^2}}\leq R\leq \frac{\ell\sin\epsilon}{\cos\epsilon-\sigma}+\frac{d\cos\epsilon}{2\sqrt{\cos^2\epsilon-\sigma^2}} .
		\end{equation} 
		It is not hard to check that the right-hand side of (\ref{lemma3.4.1}) attains the maximum at the lower bound of $R$. Thus a uniform upper bound follows from (\ref{lemma3.4.1}) and (\ref{lemma3.4.2}).

		For the lower bound, for fixed $z\in \Omega$. Now let $B_1$ be a ball in $\mathbb{R}^{n+1}$ centered at $a=(a',-\sigma R)$ of radius R such that $B_1\cap P_\epsilon=D_{d(z)}(\tilde{z})$. By (3) of Lemma \ref{geometricmaximumprinciple}, $\Sigma$ is completely outside of $B_1$. Then
		\[u(z)\geq \epsilon+\arctan\left(\frac{R}{\ell(z)}(1-\frac{\sigma}{\cos\epsilon})-\tan\epsilon\right)\]
		where, by Euclidean geometry again, $R$ satisfies 
		\[(\ell(z)\tan\epsilon+\sigma R/\cos\epsilon )^2+(d(z)/2)^2=R^2\]
		which yields
		$$R\geq \frac{\sigma \ell(z)\sin\epsilon}{\cos^2\epsilon-\sigma^2}+\frac{d(z)\cos\epsilon}{2\sqrt{\cos^2\epsilon-\sigma^2}}.$$
		This completes the proof of the lower bound.

	\end{proof}
	
	As mentioned in the introduction, we also need the angle to be sufficiently small to have gradient estimates on the boundary.
	\begin{lemma}\label{boundarygradientestimate}
		Assume $\p\Omega$ is $C^2$. Then there exists $\epsilon>0$ sufficiently small, a positive constant $C_1=C_1(\Omega,r_2,\epsilon_0,\sigma)$, a negative constant $C_2=C_2(\Omega,r_1,\epsilon_0,\sigma)$ such that for any $\epsilon<\epsilon_0$
		
		\[C_2\leq \sigma\cot u+\frac{\nu\cdot W}{\sin u}\leq C_1, \ \ \ \text{on}\ \ \p\Sigma\]
		where $r_1, r_2$ are the maximal radii of interior and exterior spheres to $\partial\Omega$, respectively. In particular, as $\epsilon$ approaches to zero, $\sigma\cos u+\nu\cdot W$ tends to zero.
		
	\end{lemma}
	
	\begin{proof}
		Assume $r_2<\infty$.  Fix a boundary point $z\in\partial\Omega$. Let $B_2$ be the ball of radius $R_2$ centered at $a=(a',\sigma R)$ such that $B_2\cap P_\epsilon$ is an n-dimensional ball of radius $r_2$ externally tangent to $\partial\Omega_\epsilon$ at $\tilde{z}$. By $(4)$ of Lemma \ref{geometricmaximumprinciple}, $\Sigma\cap B_2=\emptyset$. Hence, at $\tilde{z}$,
		\[\nu\cdot (-W)\geq \frac{\sqrt{R_2^2-r_2^2}}{R_2}\]
		where we recall $\nu$ is the unit normal of $\Sigma$ pointing outward and $W$ is defined in (\ref{wdefinition}). $R_2$ satisfies
		\begin{equation}\label{equationforR_2}
			\sigma R_2-\tilde{z}_{n+1}=\sqrt{R^2_2-r^2_2}\cos\beta+r_2\sin\beta
		\end{equation}
		where $\tilde{z}_{n+1}=\tilde{z}\cdot E_{n+1}$ is the height of point $\tilde{z}$. $\beta$ is the angle between the conormal vector of $\Omega_\epsilon\subset P_\epsilon$ at $\tilde{z}$ and horizontal hyperplane through $\tilde{z}$, thus $\beta\in [-\epsilon,\epsilon].$ Hence
		\[\nu\cdot(-W)\geq \frac{\sigma R_2-\tilde{z}_{n+1}-r_2\sin\beta}{R_2\cos\beta}.\]
		It follows that on $\p\Omega,$
		\begin{align*}
			\sigma\cot u+\frac{\nu\cdot W}{\sin u}&\leq \frac{1}{\sin\epsilon}\frac{\tilde{z}_{n+1}-r_2\sin\beta}{R_2\cos\beta}\\
			&\leq \frac{\tilde{z}_n+r_2}{R_2\cos\beta}.
		\end{align*}
		Solving (\ref{equationforR_2}) gives
		\[R_2=\frac{\tilde{z}_{n+1}^2+r_2^2+2r_2\tilde{z}_{n+1}\sin\beta}{\sigma \tilde{z}_{n+1}+\sigma r_2\sin\beta+\cos\beta\sqrt{\tilde{z}^2_{n+1}+r_2^2-r^2_2\sigma^2+2r_2\tilde{z}_{n+1}\sin\beta}}\]
		Combining them yields
		\[\sigma\cot u+\frac{\nu\cdot W}{\sin u}\leq \frac{(\tilde{z}_n+r_2)(\tilde{z}_n(\sigma/\cos\epsilon_0+1)+r_2\sqrt{1-\sigma^2})}{r_2^2-2\tilde{z}_nr_2\sin^2\epsilon_0}.\]
		As $\Omega$ is a bounded domain, we get the desired upper estimate.
		
		On the other hand, we shall deal with the interior sphere to $\partial\Omega$ at point $\tilde{z}$. We obtain
		\[\nu\cdot(-W)\leq \frac{\sqrt{R_1^2-r_1^2}}{R_1}.\]
		And $R_1$ satisfies
		$$\cos\beta \sqrt{R_1^2-r_1^2}=z_n(q)+\sigma R_1$$
		where $q$ is the center of n-ball internally tangent to $\partial\Omega_\epsilon$ at $\tilde{z}.$ 
		In a similar manner, we are able to find a negative constant $C_2$ depending on $\epsilon_0,\sigma, r_1$ and $\Omega$ such that
		\[\sigma\cot u+\frac{\nu\cdot W}{\sin u}\geq C_2.\]
	\end{proof}
	
	In order to get the gradient estimates, we need the following proposition.

	\begin{proposition}\label{maximumpricipleforgradient}
		Let $\Sigma$ be a smooth strictly convex hypersurface in $\h$ satisfying $f(\kappa)=\sigma$ where $f\in \mathcal{C}_n$. If $\Sigma$ can be locally written as a geodesic graph over a domain $\Omega\subset \mathbb{R}^n_+$, that is 
		$\Sigma=\{(z,u(z)): \ z\in\Omega\}$. Then
		\[F^{ij}\nabla_{ij}(\sigma\cot u+\frac{\nu\cdot W}{\sin u})\geq 0.\]
		
	\end{proposition}
	\begin{proof}
		Notice that $(\nu\cdot W)/\sin u=\nu_n-\nu_{n+1}\cot u$. It follows from Theorem \ref{gradientcalculation} that
		\begin{align*}
			F^{ij}\nabla_{ij}(\sigma\cot u+\frac{\nu\cdot W}{\sin u})&=\sigma^2(\nu_n-\nu_{n+1}\cot u)+\sigma\cot u\sum f_i
			-\sigma \cot u\\ &-(\nu_n-\nun\cot u )\sum f_i\kappa_i^2\\
			&=(\nu_{n+1}\cot u-\nu_n)(\sum f_i\kappa_i^2-\sigma^2)+\sigma\cot u(\sum f_i-1)\\
			&\geq (\sum f_i-1)\left(\sigma\cot u+\frac{\sigma^2(\nu_n-\nu_{n+1}\cot u)}{\sum f_i}\right)\\
			&\geq \frac{\sigma}{\sin u}(\sum f_i-1)(\cos u-\sigma)
		\end{align*}
		where we have used  $(\Sigma f_i\kappa_i^2)(\Sigma f_i)\geq \sigma^2$ in the first inequality. To ensure it is nonnegative, it suffices to show that $\cos u\geq \sigma$. In fact, this is an easy comparison between $\sigma$  and the equation (\ref{lemma3.4.1}) in the proof of Lemma \ref{c0estimate}. 
	\end{proof}
	
	By above proposition and gradient estimates on the boundary, we can obtain global gradient estimates.
	\begin{theorem}\label{c1 estimates}
		Let $\Sigma$ be a smooth strictly convex hypersurface in $\h$ satisfying $f(\kappa)=\sigma$ where $f\in \mathcal{C}_n$. Then there exists a constant $C_3$ such that for sufficiently small $\epsilon$,
		\[|\nabla u|\leq C_3, \ \ \ \text{in}\ \Omega\]
		where $C$ is independent of $\epsilon.$
	\end{theorem}
	
	\begin{proof}
		Let $g=\nu_n-\nu_{n+1}\cot u$. Then 
		\begin{equation}\label{gradientofg}
			g_i=-\kappa_i \cdot (\cot u)_i
		\end{equation}
		for each $i=1,\cdots,n.$
		Suppose $-1/g$ attains the maximum at an interior point $z_0.$ Then at $z_0$
		\[-\nabla_i \frac{1}{g}=-\frac{\kappa_i \cdot (\cot u)_i}{g^2}=0.\]
		Since $\kappa_i$ for $i=1,\cdots,n$ are positive, Then we have $\nabla\cot u=0$, and thus $\nabla u=0.$
		
		So, Lemma \ref{boundarygradientestimate} tells that
		\begin{align}\label{11111}
			-\frac{\sin u}{\nu\cdot W}&\leq \max\left\{\max_{\p \Omega} -\frac{\sin\epsilon}{\nu\cdot W}, \ \ \max_{\Omega} \sin u\right\}\nonumber\\
			&\leq 
			\max\left\{\max_{\p \Omega} \frac{\sin\epsilon}{\sigma\cos \epsilon-C_1\sin\epsilon}, \ \ \max_{\Omega} \sin u\right\}\\
			&\leq \max_\Omega\sin u \nonumber.
		\end{align}
		for sufficiently small $\epsilon.$ Thus $-\nu\cdot W\geq \sin u/\max_\Omega\sin u$. On the other hand, Proposition \ref{maximumpricipleforgradient} and Lemma \ref{boundarygradientestimate} yield that
		\[-\nu\cdot W\geq \sigma\cos u-C_1\sin u.\]
		If $\tan u\leq c$, then $-\nu\cdot W\geq \sigma\cos u-c C_1 \cos u;$ If $\tan u\geq c$, then $-\nu\cdot W\geq c\cos u/(\max_\Omega \sin u)$. Therefore choosing $c=\frac{\sigma \max_\Omega\sin u}{C_1\max_\Omega \sin u+1}$ gives to 
		\[-\nu\cdot W\geq \frac{\sigma\cos u}{C_1\max_{\Omega}\sin u+1}.\]
		This together with Lemma \ref{c0estimate} completes the proof .
	\end{proof}

	\section{Bounds for second derivatives on the boundary}\label{section4 boundary estimates}
	
	In this section, we will derive boundary estimates for the second derivatives of admissible solutions to the Dirichlet problem presented below.
	\begin{align}\label{equationinsection4}
		\begin{cases}
			G(D^2u,Du,u,y)=\sigma &\text{in} \ \Omega,\\
			u=\epsilon &\text{on}\ \p\Omega,
		\end{cases}
	\end{align}
	where $\Omega$ is a bounded domain in $\mathbb{R}^n_+$ and $G$ is defined in (\ref{definitionofG}).

	\begin{theorem}\label{boundaryestiamtesofsecondderivatives}
		Let $\Omega$ be a bounded domain in $\mathbb{R}_+^n$ with $C^3$ boundary. Suppose that $f\in \mathcal{S}_n$ . Let $u\in C^3(\bar{\Omega})$ be a admissible solution to (\ref{equationinsection4}). Then
		\[\sin u\cdot |D^2u|\leq C, \ \ \text{on}\ \p\Omega\]
		where $C$ is independent of $\epsilon.$
	\end{theorem}
	
	To prove the boundary estimates, we choose an arbitrary point $p\in\p \Omega$. Let $x_1,\cdots, x_n$ be the local coordinates such that $x(p)=0$ and the positive $x_n$ direction is the interior normal to $\p\Omega$ at $p$.  Let $\tau_i=\frac{\p}{\p x_i}$ be the orthonormal basis so $\tau_i$ for $i=1,\cdots, n-1$ are tangential to $\partial\Omega$ at $p$ and $\tau_n$ is the interior unit normal of $\p\Omega$ at $p$. Denote by $B$ the orthogonal matrix that transforms $\{E_i\}$ to $\{\tau_i\}$, that is $\tau_i=b_{ji}E_j$. We can write any point $x\in B_\delta(p)\cap \Omega$ in $y$-coordinates by $z_k=p_k+b_{ki}x_i.$
	
	We now parametrize $\Sigma$ locally in $B_\delta(p)\cap \Omega$ and calculate geometric quantities by using coordinates $\{x_i\}.$  In details, the (Euclidean) unit normal and the induced (Euclidean) Riemannian metric are 
	\[\nu=\frac{1}{\omega}(-y_nb_{1j}u_j,\cdots, -y_nb_{n-1,j}u_j,-\sin u-y_n\cos ub_{nj}u_j,\cos u-y_n\sin ub_{nj}u_j),\]
	
	\[g^E_{ij}=\delta_{ij}+y_n^2u_iu_j\]
	where $u_i=\frac{\p u}{\p x_i}$ and hereafter all the derivatives will be with respect to $\{x_i\}$. Furthermore, the Euclidean second fundamental form is 
	\[h_{ij}^E=\frac{b_{nj}u_i+b_{ni}u_j+y_n^2u_iu_jb_{nk}u_k+y_nu_{ij}}
	{\omega}.\]
	Similarly as in section \ref{parametrizationofgraph}, we have
	\begin{align*}
		a_{ij}&=\nu_{n+1}\delta_{ij}+y_{n}\sin u\gamma^{ik}h^E_{k\ell}\gamma^{\ell j}\\
		&=\frac{y_n^2\sin u}{\omega}\gamma^{ik}u_{kl}\gamma^{\ell j}+\frac{y^3_n\sin u}{\omega^3}u_iu_jb_{nk}u_k+\frac{y_n\sin u}{\omega^2}(b_{nk}\gamma^{ik}u_j+b_{n\ell}\gamma^{j\ell}u_i)\\
		&+\nu_{n+1}\delta_{ij}
	\end{align*}
	where $\gamma^{ij}=\delta_{ij}-\frac{y_n^2u_iu_i}{\omega(1+\omega)}.$
	
	We consider the linearized operator of $G$ at $u$:
	\[\mathcal{L}=G^{st}\partial_{st}+G^s\p_s+G_u,\]
	and the partial linearized operator at $u$:
	\[L=G^{st}\partial_{st}+G^s\p_s.\]
	The ambient isometries with respect to the point $p$ will produce some kernels of the operator $\mathcal{L}$. 
	
	\begin{proposition}\label{kernels}
		For any $1\leq k,\ell\leq n,$ the followings hold
		\[\mathcal{L}\left(u_k+\frac{b_{n k}}{p_n+b_{nj}x_j}\sin u\cos u\right)=0,\]
		and
		\[\mathcal{L}\left(u_kx_\ell-u_\ell x_k+\frac{b_{nk}x_\ell-b_{n\ell}x_k}{p_n+b_{nj}x_j}\sin u\cos u\right)=0.\]
	\end{proposition}
	\begin{proof}
		Consider translation along $x_k$-direction by distance $t$. Let $y=(\cdots,x_k+t,\cdots)$ and 
		\[\cot v(y)=\cot u(x)\left(1+t\frac{b_{n k}}{p_n+b_{n j}x_j}\right).\]
		Since translations are isometries in hyperbolic space, we have
		\[G(D^2v(y), Dv(y),v(y), y)=\sigma.\]
		Differentiating it with respect to $t$ and evaluating at $t=0$ yield 
		\[G^{st}\dot{v}_{st}+G^s\dot{v}_s+G_u\dot{v}=0\]
		By standard computation $\dot{v}$ is
		\[\dot{v}=-u_k-\frac{b_{n k}}{p_n+b_{nj}x_j}\sin u\cos u\]
		where $p_n=p\cdot E_n$. 
		
		On the other hand, consider rotations with respect to the vertical geodesic through the point $p$. Let $R(\theta)$ be the rotation matrix that preserves the subspace $V=\{a\tau_k+b\tau_\ell\}.$ Now let $y=R(\theta)x$, that is $$\Bar{x}=(\cdots,\cos\theta x_k-\sin\theta x_\ell,\cdots, \sin\theta x_k+\cos\theta x_\ell,\cdots).$$
		Let 
		\[\cot v(\Bar{x})=\cot u(x)\left(1+\frac{a_{nk}x_k+a_{n\ell}x_\ell}{p_n+a_{nj}x_j}(\cos\theta-1)-\frac{a_{nk}x_\ell+a_{n\ell}x_k}{p_n+a_{nj}x_j}\sin\theta\right)\]
		Again rotations are hyperbolic isometries, so\[G(D^2v(\Bar{x}), Dv(\Bar{x}),v(\Bar{x}), \Bar{x})=\sigma.\]
		Differentiating it with respect to $t$ and evaluating at $t=0$ yield 
		\[G^{st}\dot{v}_{st}+G^s\dot{v}_s+G_u\dot{v}=0\]
		By standard computation $\dot{v}$ is
		\[\dot{v}=u_ku_\ell-u_\ell x_k+\frac{a_{nk}x_\ell-a_{n\ell}x_k}{p_n+a_{nj}x_j}\sin u\cos u.\]
		This completes the proof of the Proposition.
		
	\end{proof}
	
	\begin{remark}
		The rotations we consider in above proposition preserve the infinity of $\h$. There also exist rotations at $p$ altering the infinity which induce different kernels of $\mathcal{L}$. However, above two types are enough for our applications.
	\end{remark}
	
	For later use, we need to understand the coefficients of $\mathcal{L}$. First, we see that
	\begin{align}\label{calculationofGu}
		G_u&=F^{ij}\frac{\p a_{ij}}{\p u}\nonumber\\
		&=F^{ij}(\nu_n\delta_{ij}+y_n\cos u\gamma^{ik}h^E_{k\ell}\gamma^{\ell j})\nonumber\\
		&=(\nu_n-\nu_{n+1}\cos u)\sum F^{ii}+\sigma\cot u.
	\end{align}
	Thus 
	\begin{equation}\label{estimateofGu}
		|G_u|\leq \frac{C}{\sin u}(1+\sum F^{ii}).   
	\end{equation}
	
	For $G^{st}$, we get $$G^{st}=F^{ij}\frac{y_n^2\sin u}{\omega}\gamma^{is}\gamma^{tj}.$$ Consequently,
	\begin{equation}\label{calculationofGst}
		G^{st}u_{st}=\sigma-\nu_{n+1}\sum F^{ii}-\frac{y^3_n\sin u}{\omega^3}F^{ij}u_iu_jb_{nk}u_k-\frac{2y_n\sin u}{\omega^2}F^{ij}b_{nk}\gamma^{ik}u_j.
	\end{equation}
	
	Furthermore, similar to the Lemma 2.3 in \cite{guancpam2004} we can obtain
	\begin{proposition}
		It holds:
		\[\begin{aligned} G^s&= -\frac{y_n^2}{w^2} u_s \sigma-\frac{\sin u y_n \delta_{n s}}{w} \sum F^{i i}+\frac{y_n^5 \sin u}{w^5} b_{nk} u_k u_sF^{i j} u_i u_j +\frac{2 y_n^3\sin u }{w^4 }u_s F^{i j} b_{n k} \gamma^{i k} u_j\\
			& -2y_n^2 F^{i j}  a_{i k} \frac{w u_k \gamma^{s j}+u_j \gamma^{k s}}{(1+w) w}+2\nu_{n+1}y_n^2 F^{ij}\frac{w u_i \gamma^{s j}+u_j \gamma^{i s}}{(1+w) w} \\
			&+2\frac{y_n^5\sin u}{\omega^3}F^{ij}u_i u_k b_{nl} u_l \frac{w u_k \gamma^{s j}+u_j \gamma^{k s}}{(1+w) w}
			\\ & +2 \frac{y_n^3\sin u}{w^2} F^{i j}  b_{n \ell}  \gamma^{i \ell} u_k \frac{w u_k \gamma^{s j}+u_j \gamma^{k s}}{w(1+w)}+2 \frac{y_n^3 \sin u}{w^2} 
			F^{i j} b_{n \ell} \gamma^{\ell k} u_i \frac{w u_k \gamma^{s i}+u_j \gamma^{k s}}{w(1+w)}\\
			& -\frac{3 y_n^5 \sin u}{w^5} b_{nk} u_k u_s F^{i j} u_{i} u_j+\frac{y^2_n\sin u}{w^3}\left[F^{i j} u_i u_j y_n \delta_{n s}+2 y_n b_{nk} u_k F^{i j} \delta_{i s} u_j\right] \\ & -\frac{4 y_n^3\sin u}{w^4} u_s F^{i j} b_{n k} \gamma^{i k} u_j+2 \frac{y_{n}\sin u}{w^2} F^{i j} b_{n k} \frac{\partial\left(\gamma^{i k} u_j\right)}{\p u_s}.\end{aligned}\]
		
	\end{proposition}
	
	The only term involving $a_{ij}$ is the first term in second line. Since $F^{ij}$ and $\{a_{ij}\}$ are both positive definite and can be diagonalized simultaneously by an orthogonal matrix, we have that
	\[|F^{ij}a_{ik}|\leq \sum f_i\kappa_i=\sigma.\]
	Moreover, the last term is
	\begin{flalign}\label{lasttermderivative}
		\begin{aligned}
			\frac{\partial\left(\gamma^{i k} u_j\right)}{\partial u_s}&=\delta_{i k} \delta_{j s} -\frac{y_n^2\left(\delta_{i s} u_k u_j+u_i \delta_{k s} u_j+u_i u_k \delta_{js}\right)}{w(1+w)} \\ & +\frac{y_n^2 u_i u_k u_j}{w^2(1+w)^2}\left(2 w \frac{\partial w}{\partial u_s}+\frac{\partial w}{\partial u_s}\right).
		\end{aligned} 
	\end{flalign}
	Thus it follows from the boundedness of $u,y_n,\omega, b_{ij}$ and $|\nabla u|$ that
	\begin{equation}\label{boundforGs}
		|G^s|\leq C(1+\sum F^{ii})
	\end{equation}
	where $C$ is a uniform constant.
	
	We now want to calculate $Lu$. First from (\ref{lasttermderivative}), we obtain
	\[\frac{\partial\left(\gamma^{i k} u_j\right)}{\partial u_s}u_s=\delta_{ik}u_j-\frac{\omega^2+\omega+1}{\omega^3(1+\omega)}y^2_nu_iu_ju_k.\]
	Therefore by tedious but standard computation, 
	\begin{align*}
		G^su_s&=\frac{1-\omega^2}{\omega^2}\sigma-\frac{b_{nj} u_j y_n\sin u}{\omega}\sum F^{ii}-\frac{2y^2_n}{\omega^2}F^{ij}a_{ik}u_ju_k+\frac{2y_n\sin u}{\omega^2}F^{ij}b_{ni}u_j\\
		&+F^{ij}u_iu_j\left(\frac{-2y_n^3\sin ub_{n\ell}u_\ell}{\omega^3(1+\omega)}+\frac{2\nu_{n+1}y_n^2}{\omega^2}+\frac{3 b_{nk} u_k y_n^3\sin u }{\omega^3}\right).
	\end{align*}
	This together with (\ref{calculationofGst}) yields
	\begin{equation}\label{calculationofLu}
		Lu=\frac{\sigma}{\omega^2}-\frac{2y^2_n}{\omega^2}F^{ij}a_{ik}u_ju_k-(\nu_{n+1}+\frac{b_{nj} u_j y_n\sin u}{\omega})\sum F^{ii}+\frac{2\cos u y_n^2}{\omega^3}F^{ij}u_iu_j.
	\end{equation}

	We will employ the following lemma in establishing the boundary estimates of second derivatives.
	\begin{lemma}\label{calculationofLcotu}
		Suppose that $f\in \mathcal{C}_n$ . Then
		\[L(1-\tan\epsilon\cot u)\leq -\frac{\tan\epsilon(\omega\cos u-\sigma)}{\omega^2\sin^2 u}(1+\sum F^{ii}).\]
		
	\end{lemma}

	\begin{proof}
		It follows from (\ref{calculationofLu}) that
		\begin{align*}
			L\cot u&= -\frac{1}{\sin ^2u}Lu+\frac{2y^2_n\cot u}{\omega^3\sin u }F^{ij}u_iu_j\\
			&=\frac{\sigma}{\omega^2\sin^2u}+\frac{2y^2_n}{\omega^2\sin^2u}F^{ij}a_{ik}u_ju_k+\frac{1}{\sin^2u}(\nu_{n+1}+\frac{b_{nj} u_j y_n\sin u}{\omega})\sum F^{ii}\\
			&\geq \frac{\sigma}{\omega^2\sin^2u}+\frac{\cos u}{\omega\sin^2u}\sum F^{ii}
		\end{align*}
		where we have used the expression of $\nun$ and the fact that 
		\[F^{ij}a_{ik}\zeta_k\zeta_j\geq 0, \ \ \forall \ \zeta\in \mathbb{R}^n.\]
		
		Applying the fundamental inequality $2(a-b)\geq (1-b)(1+a)$ for any $a\geq 1$ and $b\geq 0$, we obtain
		\[L\cot u\geq \frac{\cos u}{2\omega \sin^2 u}(1-\frac{\sigma}{\omega \cos u})(1+\sum F^{ii})\geq 0.\]
		Therefore, the lemma follows from above inequality.
		
	\end{proof}

	\begin{proof}[Proof of Theorem \ref{boundaryestiamtesofsecondderivatives}]
		As mentioned in the beginning of this section, we derive the estimates of second derivatives at the point $p\in \p\Omega.$
		Since the compact boundary is assumed to be $C^3$, there exists a uniform constant $\delta>0$ such that $\partial\Omega\cap B_\delta(p)$ can be represented as a graph:
		\[x_n=\rho(x')=\frac{1}{2}\sum_{\alpha,\beta<n}B_{\alpha\beta}x_\alpha x _\beta+O(|x'|^3),\ \ \ x'=(x_1,\cdots,x_{n-1})\]
		where $B_{\alpha\beta}$ are the principal curvatures of $\partial\Omega$ at $p$.
		The boundedness of $u_{\alpha\beta}(p)$ follows in a standard way. In fact, $u(x',\rho(x'))=\epsilon$ along $\p\Omega$. Then 
		for $\alpha,\beta<n$ we have $u_{\alpha\beta}(p)=-u_n(p)B_{\alpha\beta}$. Thus
		\[|u_{\alpha\beta}|(p)\leq C\]
		where $C$ depends on the geometry of $\p\Omega$ and the upper bound of $|\nabla u|.$
		
		For estimates of second derivatives of other mixed indices, as in \cite{caffarelli1984},\cite{guan2009jga},\cite{guanjems} we consider for fixed $\alpha<n$ the operator
		\begin{equation*}
			T=\p_\alpha+\sum_{\beta<n}B_{\alpha\beta}(x_\beta\p_n-x_n\p_\beta).
		\end{equation*}
		There exists a constant $C$ such that
		\begin{equation}\label{interiorboundforTu}
			|Tu|\leq  C,\ \ \ \ \text{in}\ B_\delta(p)\cap \Omega,
		\end{equation}
		and since $u_\alpha+\sum_{\beta<n}B_{\alpha\beta}x_\beta u_n=u_n(-\rho_\alpha+\sum_{\beta<n}B_{\alpha\beta}x_\beta)$ on $\p \Omega$ around $p$, then
		\begin{equation}\label{boundaryboundforTu}
			|Tu|\leq C|x|^2, \ \ \ \ \text{on}\ \p\Omega\cap B_\delta(p).
		\end{equation}
		
		Define a function
		\[\phi_1=A(1-\tan\epsilon \cot u)+B|x|^2,\]
		and 
		\[\phi_2=Tu+\frac{b_{n\alpha}}{2y_n}(\sin 2u-\sin 2\epsilon)-\sum_{\beta<n}B_{\alpha\beta}\frac{b_{n\beta}x_n-b_{nn}x_\beta}{2y_n}(\sin 2u-\sin 2\epsilon).\]
		By (\ref{interiorboundforTu}),(\ref{boundaryboundforTu}), $|\phi_2|\leq C|x|^2$ on $\p\Omega\cap B_\delta(p)$ and  $|\phi_2|\leq C$ in $B_\delta(p)\cap \Omega$ with a strictly larger constant $C$ than the original one.
		Let
		\[\phi=\phi_1\pm \phi_2.\]
		Therefore, we can choose $B=C/\delta$ such that $\phi\geq 0$ on $\partial(\Omega\cap B_\delta(p)).$ On the other hand, Proposition \ref{kernels} and (\ref{boundforGs}) together with simple computation imply 
		\[|L\phi_2|\leq C_1(1+\sum F^{ii})\]
		for a uniform constant $C_1>0.$
		Hence, it follows from (\ref{boundforGs}) and Lemma \ref{calculationofLcotu} $\Omega\cap B_\delta(p),$
		\[L\phi\leq \left(-\frac{\sin\epsilon}{\omega^2}(\omega\cos u-\sigma)\frac{A}{\sin^2 u}+C\cdot B+C_1\right)(1+\sum F^{ii})\]
		which can be made to be negative by choosing large enough $A$.
		
		Thus by maximum principle, $\phi\geq 0$ in $\Omega\cap B_\delta(p)$. Since $\phi(p)=0$, so $p$ is a minimum point thus $\phi_n(p)\geq 0$, that is, at $p$, 
		\[\frac{A}{\cos\epsilon\sin\epsilon}u_n\pm \left(u_{\alpha n}+\frac{b_{n\alpha}\cos2\epsilon }{y_n}u_n\right)\geq 0.\]
		Because $u_n\geq 0$, this gives
		\[\sin\epsilon \cdot|u_{\alpha n}(p)|\leq C |u_n(p)|.\]
		
		It remains to show the boundedness of $u_{nn}(p)$. The proof follows basically same lines as in \cite[p.786]{guan2009jga}. For completeness, we include the proof here. We assume $u_{\alpha,\beta}$ for $1\leq \alpha, \beta\leq n-1$ is diagonal. Denote $\ell_\epsilon=\cos\epsilon-y_n\sin\epsilon$. Then at $p$,
		\[[a_{ij}]=\frac{1}{\omega}
		\begin{bmatrix}
			
			\ell_\epsilon+y^2_n\sin\epsilon u_{11}&0& \ldots & \frac{y_n^2\sin\epsilon}{\omega}u_{1n}\\
			0&\ddots&\cdots& \frac{y_n^2\sin\epsilon}{\omega}u_{2n} \\
			\vdots&\vdots&\ddots&\vdots\\ 
			\frac{y_n^2\sin\epsilon}{\omega}u_{1n}&\frac{y_n^2\sin\epsilon}{\omega}u_{2n}&\ldots& \ell_\epsilon+\frac{y_n\sin\epsilon (\omega^2+1)}{\omega^2}u_n+\frac{y_n^2\sin\epsilon }{\omega^2}u_{nn}.
		\end{bmatrix} 
		\]
		The Lemma 1.2 in \cite{caffarelli1984} tells that if $\sin\epsilon u_{nn}(p)$ tends to infinity, i.e., $a_{nn}\rightarrow \infty$, then eigenvalues of $[a_{ij}]$ are asymptotically given by 
		\[\kappa_\alpha=\frac{1}{\omega}(\ell_\epsilon+y^2_n\sin\epsilon u_{11}+o(1))\ \ \alpha\leq n-1\]
		and
		\[\kappa_n=a_{nn}(1+O(\frac{1}{a_{nn}}))\]
		where the asymptotic constants only depend on bound of $u_{\alpha\beta}.$
		For $\epsilon$ sufficiently small such that
		\[(\omega\kappa_1,\cdots,\omega\kappa_{n-1}, \ell_\epsilon+\frac{y_n\sin\epsilon (\omega^2+1)}{\omega^2}u_n)\]
		is uniformly closed to $(1,\cdots, 1)$. 
		Then we can use conditions (\ref{f degree 1}), (\ref{f unn condition}) to obtain
		\[\sigma=\frac{1}{\omega} F(\omega[a_{ij}])\geq \frac{1}{\omega}(1+\epsilon_0)\]
		if $\sin\epsilon u_{nn}(p)\geq R_0$ for a large but uniform constant $R_0$. This inequality implies a contradiction by observing 
		\[\sigma\geq (-\nu\cdot W)(1+\epsilon_0)\geq (\sigma\cos\epsilon-C_1\sin\epsilon)>\sigma.\]
		for sufficiently small $\epsilon$. We used (\ref{11111}) in the second inequality. Thus 
		\[\sin\epsilon\cdot u_{nn}(p)\leq R_0\]
		completing the proof.
		
	\end{proof}
	
	\section{curvature estimates and higher order estimates}\label{section5}
	In this section we prove a curvature estimate for compact strictly convex hypersurfaces satisfying $f(\kappa)=\sigma$. By allowing its boundary to approach to the infinity, we also prove a global interior curvature estimate for the largest principal curvature of strictly convex graphs with uniformly bounded principal curvatures satisfying $f(\kappa)=\sigma$.
	
	For a fixed point $x_0\in \Sigma$, choosing a local orthonormal frame $\{e_i\}_{i=1,\cdots,n}$ around $x_0$ such that $h_{ij}(x_0)=\kappa_i\delta_{ij}$, then in $\mathbb{H}^{n+1}$ there holds Codazzi equation $h_{ijk}=h_{ikj}$. Ricci identity and Gauss equation imply the following commutator formula.
	\begin{equation}
		h_{iijj}-h_{jjii}=(\kappa_i\kappa_j-1)(\kappa_i-\kappa_j).
	\end{equation}
	Thus
	\begin{equation}\label{com formula}
		F^{ii}h_{jjii}=F^{ii}h_{iijj}+(1+\kappa_j^2)\sum_{i=1}^{n}-\kappa_j\sum_{i=1}^{n}f_i-\kappa_j\sum_{i=1}^n\kappa_i^2f_i.
	\end{equation}
	\begin{theorem}\label{interior C2 esitimates thm}
		Let $\Sigma$ be a smooth strictly convex hypersurface in $\mathbb{H}^{n+1}$ satisfying $f(\kappa)=\sigma$ and suppose that there exist constants $C_1$, $C_2>1$ and $0<a<1$ such that
		\begin{equation}\label{C0 upper bound condition}
			\tan u\leq C_1 \ \ \ \ \text{on}\ \Sigma,
		\end{equation}
		\begin{equation}\label{C1 lower bound condition}
			2a\leq -\frac{\nu\cdot W}{\cos u} \leq C_2\ \ \ \ \text{on}\ \Sigma
		\end{equation}
		where $\cot u=y_n/y_{n+1}$ and $z=(y_1,\cdots,y_n,y_{n+1})$ is the position vector.
		Let $\kappa_{\max}(p)$ be the largest principal curvature of $\Sigma$ at $p$. Then 
		\begin{enumerate}
			\item[(1)] If $\Sigma$ is compact, we have \begin{equation}\label{curvature estimates}
				\underset{\Sigma}{\max}\frac{\kappa_{\max}}{-\nu\cdot W \sec u-a}\le \max\left\{ \frac{C_3}{a},\underset{\partial\Sigma}{\max}\frac{\kappa_{\max}}{-\nu\cdot W \sec u-a}\right\}
			\end{equation}
			for a uniform constant $C_3$ only depending on $C_1, C_2$ and $a$.
			\vskip.1cm

			\item[(2)] If $\Sigma$ is complete noncompact with boundary $\partial\Sigma\subset \p_\infty \h$ and has bounded principal curvatures $0<\kappa_i\leq C_4$, we have that for $0<b<\frac{a}{2C_2(3+2C_1^2)},$
			\begin{equation}\label{curvature estimates1}
				\underset{\Sigma}{\sup}\frac{\tan^b u\cdot \kappa_{\max}}{-\nu\cdot W \sec u-a}\le \frac{C_3}{a} \sup_{\Sigma} \tan^b u
			\end{equation}
			for a uniform constant $C_3$ only depending on $C_1, C_2$ and $a$. In particular, letting $b$ tend to zero yields
			\begin{equation}
				\kappa_{\max} \leq \frac{C_3}{a}.
			\end{equation}
		\end{enumerate}
		
	\end{theorem}
	\vskip.2cm
	\begin{remark}
		For convex geodesic graph $\Sigma$ with boundary $\partial\Sigma\subset P_\epsilon$ satisfying $f(\kappa)=\sigma$, the required assumption are satisfied. Specifically, assumption (\ref{C0 upper bound condition}) follows from  Lemma \ref{c0estimate} while (\ref{C1 lower bound condition}) follows from Lemma \ref{c0estimate} and Theorem \ref{c1 estimates}.
	\end{remark}
	
	\vskip.2cm
	\begin{proof}[Proof of Theorem \ref{interior C2 esitimates thm}]
		Let $$G=(\tan u)^b\frac{\kappa_{\max}}{-\nu\cdot W \sec u-a}.$$
		Assume the maximum of $G$ is attained at an interior point $p_0\in\Sigma$. Choose a local orthonormal frame $\{e_i\}_{i=1,\cdots,n}$ around $p_0$ such that $h_{ij}(p_0)=\kappa_i\delta_{ij}$. Without loss of generality we may assume $\kappa_1=\kappa_{\max}(p_0)$. Therefore $(\tan u)^b\frac{h_{11}}{-\nu\cdot W \sec u-a}$ achieves its local maximum at $p_0$.
		Write $g=\nu_n-\nun \cot u$ so $-\nu\cdot W \sec u=-g\tan u.$
		Consequently, by considering critical equations of $\log G$ at $p_0$ we have
		\[
		\frac{h_{11i}}{h_{11}}+b\frac{(\tan u)_i}{\tan u}-\frac{(g\tan u)_i}{g\tan u+a}=0,\]
		and 
		\begin{equation*}
			\frac{h_{11ii}}{h_{11}}-\frac{(g\tan u)_{ii}}{g\tan u+a}+b\frac{(\tan u)_{ii}}{\tan u}-(b+b^2)\frac{(\tan u)_i^2}{(\tan u)^2}+2b\frac{(\tan u)_i(g\tan u)_i}{(\tan u)(g\tan u+a)}\leq 0.
		\end{equation*}
		Thus
		\begin{flalign}
			\begin{aligned}
				\label{inequalityofFijh11ij}
				0&\geq F^{ii}h_{11ii}-\kappa_1\frac{F^{ii}(g\tan u)_{ii}}{g\tan u+a}\\
				&+F^{ii}\left(b\frac{(\tan u)_{ii}}{\tan u}-(b+b^2)\frac{(\tan u)_i^2}{(\tan u)^2}+2b\frac{(\tan u)_i(g\tan u)_i}{(\tan u)(g\tan u+a)}\right)
			\end{aligned}
		\end{flalign}

		Differentiating $F(h_{ij})=\sigma$ twice and using (\ref{com formula}) yield
		\begin{equation}\label{exchangeofFiih11ii}
			F^{ii}h_{11ii}=-F^{ij,rs}h_{ij1}h_{rs1}+\sigma(1+\kappa_1^2)-\kappa_1\sum f_i-\kappa_1\sum\kappa_i^2f_i.
		\end{equation}
		Due to an inequality  by Andrews \cite{benandrews1994convex} and Gerhardt \cite{gerhardt1996} which states
		\[-F^{ij,rs}h_{ij1}h_{rs1}\geq \sum_{i\neq j}\frac{f_i-f_j}{\kappa_j-\kappa_i}h_{ij1}^2\geq 2\sum_{i\geq 2}\frac{f_i-f_1}{\kappa_1-\kappa_i}h_{i11}^2.\]
		On the other hand, we have from (\ref{gradientofg})
		\begin{align*}
			h_{i11}&=\kappa_1\left(\frac{(g\tan u)_i}{g\tan u+a}-b\frac{(\tan u)_i}{\tan u}\right)\\
			&=\kappa_1\frac{(\cot u)_i}{\cot u} \left(-\frac{\kappa_i+g\tan u}{g\tan u+a}+b\right).
		\end{align*}
		Thus
		\begin{equation}\label{fijrshij1hrs1}
			-F^{ij,rs}h_{ij1}h_{rs1}\geq 2\kappa_1^2\sum_{i\geq 2}\frac{f_i-f_1}{\kappa_1-\kappa_i}\frac{(\cot u)^2_i}{\cot^2 u}\left(-\frac{\kappa_i+g\tan u}{g\tan u+a}+b\right)^2. 
		\end{equation}
		\vskip.2cm
		
		We now calculate $F^{ij}(g\tan u)_{ij}$. It follows
		\[F^{ij}(g/\cot u)_{ij}=F^{ij}\left(\frac{g_{ij}}{\cot u}-\frac{2g_i(\cot u)_j}{\cot^2 u}+\frac{2g(\cot u)_i(\cot u)_j}{\cot^3 u}-\frac{g(\cot u)_{ij}}{\cot^2 u}\right).\]
		Recall (\ref{t3.2.2}) and (\ref{t3.2.1}), we have
		\[F^{ij}g_{ij}=-g\sum f_i\kappa_i^2-\sigma\cot u,\]
		\[F^{ij}(\cot u)_{ij}=\sigma g+\cot u\sum f_i.\]
		Moreover, (\ref{gradientofg}) says
		\[F^{ij}g_i(\cot u)_j=-F^{ij}\kappa_i(\cot u)_i(\cot u)_j=-\sum f_i\kappa_i(\cot u)_i^2.\]
		
		Hence, combining these formulae and (\ref{inequalityofFijh11ij}),(\ref{exchangeofFiih11ii}),(\ref{fijrshij1hrs1}) imply
		
		\begin{flalign}
			\begin{aligned}\label{mainequation1}
				0&\geq 2\kappa_1^2\sum_{i\geq 2}\frac{f_i-f_1}{\kappa_1-\kappa_i}\frac{(\cot u)^2_i}{\cot^2 u}\left(-\frac{\kappa_i+g\tan u}{g\tan u+a}+b\right)^2\\
				&+\sigma\left(1+\kappa_1^2+\frac{\kappa_1}{g\tan u+a}+\frac{\kappa_1 (g\tan u)^2}{g\tan u+a}-bg\tan u \kappa_1\right)\\
				&+\frac{-a\kappa_1}{g\tan u+a}(\sum f_i+\sum f_i\kappa_i^2)-b\kappa_1\sum f_i\\
				&+\frac{(2b-2)\kappa_1}{g\tan u+a}\sum f_i\frac{(\cot u)_i^2}{\cot^2 u}(g\tan u+\kappa_i).
			\end{aligned}
		\end{flalign}
		By the assumption (\ref{C1 lower bound condition}), we can assume that $\kappa_1\geq \frac{1+C_2^2}{a}$, otherwise we are done with the proof. By the assumption, $b\leq \frac{a}{C_2}$ so that the third line of (\ref{mainequation1}) is nonegative. Therefore, all terms in (\ref{mainequation1}) are certainly nonegative except possibly the last one. From now on, we will use $C$ with lower indices to denote universal constants that only depend on $\sigma, \Omega$ and the dimension $n.$
		
		Define 
		\[J=\left\{i: g\tan u+\kappa_i<0, \ f_i\leq \frac{f_1}{\delta}\right\}\]
		\[L=\left\{i: g\tan u+\kappa_i<0, \ f_i>\frac{f_1}{\delta}\right\}\]
		where $\delta \in (0,1)$ is to be determined.
		
		Regarding the last term of (\ref{mainequation1}), considering summation in $J$ gives
		
		\begin{align*}
			\frac{(2b-2)\kappa_1}{g\tan u+a}\sum_{J}f_i\frac{(\cot u)_i^2}{\cot^2 u}(g\tan u+\kappa_i)&\geq \frac{(2-2b)\kappa_1 f_1}{a\delta}\sum_{J}(g\tan u+\kappa_i)\frac{(\cot u)_i^2}{\cot^2 u}\\
			&\geq \frac{(2-2b)\kappa_1 f_1 g\tan u}{a\delta}\sum_{J}\frac{(\cot u)_i^2}{\cot^2 u}\\
			&\geq  -\frac{(2-2b)C_2\sigma(1+C_1^2)}{a\delta}
		\end{align*}
		where in the last inequality we have used that $\kappa_1f_1\leq \sigma$ and
		\[\sum \frac{(\cot u)_i^2}{\cot^2 u}=1+\tan^2 u-(g\tan u)^2\leq 1+C_1^2\]
		followed by the the assumption (\ref{C0 upper bound condition}). 
		
		For the summation in $L$ of last term, we absorb it by the summation in $L$ of first term. Indeed,
		\begin{align*}
			&2\kappa_1^2\sum_{i\geq 2}\frac{f_i-f_1}{\kappa_1-\kappa_i}\frac{(\cot u)^2_i}{\cot^2 u}\left(-\frac{\kappa_i+g\tan u}{g\tan u+a}+b\right)^2\\
			&\geq 2(1-\delta)\kappa_1\sum_{L}f_i\frac{(\cot u)^2_i}{\cot^2 u}\left(-\frac{\kappa_i+g\tan u}{g\tan u+a}+b\right)^2\\
			&\geq 2(1-\delta)\kappa_1 \sum_{L}f_i\frac{(\cot u)^2_i}{\cot^2 u}\left(\frac{\kappa_i+g\tan u}{g\tan u+a}\right)^2-4(1-\delta)b\kappa_1 \sum_{L}f_i\frac{(\cot u)^2_i}{\cot^2 u}\left(\frac{\kappa_i+g\tan u}{g\tan u+a}\right)
		\end{align*}
		in which there exists a term:
		\begin{align*}
			2\kappa_1&\sum_{L}f_i\frac{(\cot u)^2_i}{\cot^2 u}\left(\frac{\kappa_i+g\tan u}{g\tan u+a}\right)^2\\&= 2\kappa_1\sum_{L}f_i\frac{(\cot u)^2_i}{\cot^2 u}\left(\frac{\kappa_i+g\tan 
				u}{g\tan u+a}+\frac{\kappa_i^2+(g\tan u-a)\kappa_i-ag\tan u}{(g\tan u+a)^2}\right)\\
			&\geq 2\kappa_1\sum_{L}f_i\frac{(\cot u)^2_i}{\cot^2 u}\left(\frac{\kappa_i+g\tan 
				u}{g\tan u+a}\right)+2\kappa_1\sum_{L}f_i\frac{(\cot u)^2_i}{\cot^2 u}\frac{(g\tan u-a)\kappa_i}{(g\tan u+a)^2}\\
			&\geq 2\kappa_1\sum_{L}f_i\frac{(\cot u)^2_i}{\cot^2 u}\left(\frac{\kappa_i+g\tan 
				u}{g\tan u+a}\right)-\frac{2(C_2+a)(1+C_1^2)}{a^2}\sigma
			\kappa_1.
		\end{align*}
		whose first term eliminates the part of the last term of (\ref{mainequation1}). 
		
		Simplifying (\ref{mainequation1}) gives
		
		\begin{flalign}
			\begin{aligned}\label{mainequation2}
				0
				&\geq \sigma(1+\kappa_1^2-\frac{1+C_2^2}{a}\kappa_1)-\frac{(2-2b)C_2\sigma(1+C_1^2)}{a\delta}-\frac{2(C_2+a)(1+C_1^2)}{a^2}\sigma \kappa_1\\
				&+\frac{-a\kappa_1}{g\tan u+a}(\sum f_i+\sum f_i\kappa_i^2)-b\kappa_1\sum f_i-2\delta\kappa_1\sum_{L}f_i\frac{(\cot u)^2_i}{\cot^2 u}\left(\frac{\kappa_i+g\tan u}{g\tan u+a}\right)^2\\
				&-2b(1-2\delta)b\kappa_1 \sum_{L}f_i\frac{(\cot u)^2_i}{\cot^2 u}\left(\frac{\kappa_i+g\tan u}{g\tan u+a}\right)\\
				&\geq \sigma(1+\kappa_1^2-\frac{1+C_2^2}{a}\kappa_1)-\frac{(2-2b)C_2\sigma(1+C_1^2)}{a\delta}-\frac{2(C_2+a)(1+C_1^2)}{a^2}\sigma \kappa_1\\
				&+\frac{-a\kappa_1}{g\tan u+a}(\sum f_i+\sum f_i\kappa_i^2)-(b+2b(1+C_1^2))\kappa_1\sum f_i\\
				&-2\delta\kappa_1\sum_{L}f_i\frac{(\cot u)^2_i}{\cot^2 u}\left(\frac{\kappa_i+g\tan u}{g\tan u+a}\right)^2.
			\end{aligned}
		\end{flalign}
		
		Since $b\leq \frac{a}{2C_2(3+2C_1^2)}$, the fourth term and fifth term of right-hand side of (\ref{mainequation2}) can be written as
		\begin{align*}
			IV+V&\geq \frac{a\kappa_1}{-2(g\tan u+a)C_2^2}\Big((C_2^2-(g\tan u)^2)\sum f_i\\&+\sum (\kappa_i+g\tan u)^2 f_i+(C_2^2-1)\sum f_i\kappa_i^2-2\sigma g\tan u \Big)\\
			&\geq \frac{a\kappa_1}{-(g\tan u+a)C_2^2}\Big(\sum (\kappa_i+g\tan u)^2 f_i-2\sigma g\tan u \Big)\\
			&\geq \frac{a\kappa_1}{-(g\tan u+a)C_2^2}\sum (\kappa_i+g\tan u)^2 f_i+ \frac{4a^2}{C_2-a}\sigma \kappa_1.
		\end{align*}
		Therefore, choosing $\delta=\frac{a^2}{2(1+C_1^2)C_2^2}$, we can reduce (\ref{mainequation2}) to
		\begin{align*}
			0&\geq \sigma(1+\kappa_1^2-\frac{1+C_2^2}{a}\kappa_1)-\frac{(2-2b)C_2\sigma(1+C_1^2)}{a\delta}-\frac{2(C_2+a)(1+C_1^2)}{a^2}\sigma \kappa_1\\
			&+ \frac{4a^2}{C_2-a}\sigma \kappa_1
		\end{align*}
		Solving the equation we can find a constant $C_3=C_3(C_1,C_2,a)$ such that
		\[\kappa_1\leq C_3.\]
		thus finishing the proof of the theorem.

	\end{proof}
	
	The first part of Theorem \ref{interior C2 esitimates thm} together with Theorem \ref{boundaryestiamtesofsecondderivatives} implies that the graph associated to the solution of (\ref{newboundarycondition}) has uniform bounded curvature estimates, and thus uniform bounds for $\sin u^\epsilon\cdot |D^2 u|$ exist.
	The uniform $C^0$, $C^1$ and $C^2$ estimates tell us that the fully nonlinear PDE under consideration is uniform elliptic. Thus,
	the standard Evans-Krylov-Safonov theory \cite{Gilbargtrudingerbook} implies the higher estimates $|u|_{C^{4,\alpha}}\leq C$ for a uniform constant $C$.

	\section{Proof of the theorem \ref{thm with bdy}}\label{section 6}
	In this section we give a proof to the main theorem \ref{thm with bdy} by the degree theory.
	
	Consider for $0\leq t\leq 1$ the family of Dirichlet problems
	\begin{align}\label{equationinsection6.1}
		\begin{cases}
			G(D^2u^t,Du^t,u^t,z)=t\sigma+(1-t)\cos\epsilon &\text{in} \ \Omega,\\
			u^t=\epsilon &\text{on}\ \p\Omega,\\
			u^0\equiv \epsilon & \text{on}\ \bar{\Omega}.
		\end{cases}
	\end{align}
	Since $G_u|_{u^0}=-\sin\epsilon$, we can find a smooth family of solutions $u^t$, $0\leq t\leq t_0$ for a small time $t_0$ by the implicit function theorem. However, the linearized operator is not necessarily invertible all the time, we can not use method of continuity directly. Instead, we use the degree theory for second-order fully nonlinear elliptic operators developed in
	Y. Y. Li \cite{yanyanlidegree} for the existence proof.
	
	Let $v=u-\epsilon.$ The equation (\ref{equationinsection4}) turns to be
	
	\begin{align}\label{equationinsection6.2}
		\begin{cases}
			\tilde{G}(D^2v,Dv,v,z)=\sigma &\text{on} \ \Omega,\\
			v=0 &\text{on}\ \p\Omega,
		\end{cases}
	\end{align}
	where the linearized operator of $\tilde{G}$ is basically same as the one $G$. In particular, \begin{equation}\label{newlinearizedoperator}
		\tilde{G}_v=(\nu_n-\nu_{n+1}\cos (v+\epsilon))\sum F^{ii}+\sigma\cot(v+\epsilon).
	\end{equation}
	
	Let $E(v,t): C_0^{4,\alpha}(\bar{\Omega})\rightarrow C^{2,\alpha}(\bar{\Omega})$ be defined by
	\[E(v,t)=\tilde{G}(D^2v,Dv,v,z)-\sigma_t\]
	for $t\in [0,1]$ where $\sigma_t=t\sigma+(1-t)\cos\epsilon$. Notice that $\sigma_t$ has same bound as $\sigma$. Let $C$ be the uniform constant obtained from previous sections such that a priori solution to $E(v,t)=0$ has  a uniform upper bound $\|u\|_{C^{4,\alpha}(\bar{\Omega})}< C$ and $I/C < A[u]< CI$.  Define a large subset $O\subset C_0^{4,\alpha}(\bar{\Omega})$ by
	\[O=\{v\in C_0^{4,\alpha}(\bar{\Omega}): \frac{1}{C}\delta_{ij}<a_{ij}(v)<C\delta_{ij},\|v\|_{C^{4,\alpha}(\bar{\Omega})}< C\}.\]
	where 
	\begin{align*}
		a_{ij}(v)=\frac{z^2_nz_{n+1}}{\omega}\gamma^{ik}u_{kl}\gamma^{\ell j}+\frac{z^2_nz_{n+1}}{\omega^3}u_iu_ju_n+\frac{z_{n+1}}{\omega^2}(\gamma^{in}u_j+\gamma^{jn}u_i)+\nu_{n+1}\delta_{ij}.
	\end{align*}
	Apparently, it is an open bounded subset under norm $\|\cdot\|_{C^{4,\alpha}(\bar{\Omega})}$. Firstly, notice $E(\cdot,t)$ is not zero on $\partial\Omega$. Using the Proposition 1.2 of \cite{yanyanlidegree}, we conclude that
	\[\operatorname{degree}(E(\cdot,1),O,0)=\operatorname{degree}(E(\cdot,0),O,0).\]
	It suffices to show that $\operatorname{degree}(E(\cdot,0),O,0)\neq 0$, thus it will follow immediately that $\operatorname{degree}(E(\cdot,1),O,0)\neq 0$. That is, there exists a $v_0$ in $O$ such that $\tilde{G}(D^2v_0,Dv_0, v_0, z)=\sigma$. This will in turn give our a solution $u_0$ to the original equation. 
	
	At $t=0$, the linearized operator $L_0$ is invertible since (\ref{newlinearizedoperator}) is negative. Furthermore, at $t=0,$ the only solution to equation (\ref{equationinsection6.2}) is constant $v=\epsilon$ by the geometric maximum principle.  Thus due to Proposition 1.3 in \cite{yanyanlidegree} we have
	\[\operatorname{degree}(E,O,0)=\operatorname{degree}(L_0,O,0)\neq 0.\]
	This completes the proof of Theorem \ref{thm with bdy}.

	\bibliographystyle{plain}
	\bibliography{Reference}

\begin{bibdiv}
\begin{biblist}

\bib{anderson1982}{article}{
      author={Anderson, Michael~T.},
       title={Complete minimal varieties in hyperbolic space},
        date={1982},
        ISSN={0020-9910},
     journal={Invent. Math.},
      volume={69},
      number={3},
       pages={477\ndash 494},
         url={https://doi.org/10.1007/BF01389365},
}

\bib{anderson1983}{article}{
      author={Anderson, Michael~T.},
       title={Complete minimal hypersurfaces in hyperbolic {$n$}-manifolds},
        date={1983},
        ISSN={0010-2571},
     journal={Comment. Math. Helv.},
      volume={58},
      number={2},
       pages={264\ndash 290},
         url={https://doi.org/10.1007/BF02564636},
}

\bib{benandrews1994convex}{article}{
      author={Andrews, Ben},
       title={Contraction of convex hypersurfaces in {E}uclidean space},
        date={1994},
        ISSN={0944-2669},
     journal={Calc. Var. Partial Differential Equations},
      volume={2},
      number={2},
       pages={151\ndash 171},
         url={https://doi.org/10.1007/BF01191340},
}

\bib{caffarelli1984}{article}{
      author={Caffarelli, L.},
      author={Nirenberg, L.},
      author={Spruck, J.},
       title={The {D}irichlet problem for nonlinear second-order elliptic
  equations. {I}. {M}onge-{A}mp\`ere equation},
        date={1984},
        ISSN={0010-3640},
     journal={Comm. Pure Appl. Math.},
      volume={37},
      number={3},
       pages={369\ndash 402},
         url={https://doi.org/10.1002/cpa.3160370306},
}

\bib{caffarelli1986}{incollection}{
      author={Caffarelli, L.},
      author={Nirenberg, L.},
      author={Spruck, J.},
       title={Nonlinear second order elliptic equations. {IV}. {S}tarshaped
  compact {W}eingarten hypersurfaces},
        date={1986},
   booktitle={Current topics in partial differential equations},
   publisher={Kinokuniya, Tokyo},
       pages={1\ndash 26},
}

\bib{gerhardt1996}{article}{
      author={Gerhardt, Claus},
       title={Closed {W}eingarten hypersurfaces in {R}iemannian manifolds},
        date={1996},
        ISSN={0022-040X},
     journal={J. Differential Geom.},
      volume={43},
      number={3},
       pages={612\ndash 641},
         url={http://projecteuclid.org/euclid.jdg/1214458325},
}

\bib{Gilbargtrudingerbook}{book}{
      author={Gilbarg, David},
      author={Trudinger, Neil~S.},
       title={Elliptic partial differential equations of second order},
      series={Classics in Mathematics},
   publisher={Springer-Verlag, Berlin},
        date={2001},
        ISBN={3-540-41160-7},
        note={Reprint of the 1998 edition},
}

\bib{guanspruck2000}{article}{
      author={Guan, Bo},
      author={Spruck, Joel},
       title={Hypersurfaces of constant mean curvature in hyperbolic space with
  prescribed asymptotic boundary at infinity},
        date={2000},
     journal={Amer. J. Math.},
      volume={122},
      number={5},
       pages={1039\ndash 1060},
}

\bib{guancpam2004}{article}{
      author={Guan, Bo},
      author={Spruck, Joel},
       title={Locally convex hypersurfaces of constant curvature with
  boundary},
        date={2004},
        ISSN={0010-3640},
     journal={Comm. Pure Appl. Math.},
      volume={57},
      number={10},
       pages={1311\ndash 1331},
         url={https://doi.org/10.1002/cpa.20010},
}

\bib{guanjems}{article}{
      author={Guan, Bo},
      author={Spruck, Joel},
       title={Hypersurfaces of constant curvature in hyperbolic space. {II}},
        date={2010},
        ISSN={1435-9855},
     journal={J. Eur. Math. Soc. (JEMS)},
      volume={12},
      number={3},
       pages={797\ndash 817},
         url={https://doi.org/10.4171/JEMS/215},
}

\bib{guansurvey}{incollection}{
      author={Guan, Bo},
      author={Spruck, Joel},
       title={Convex hypersurfaces of constant curvature in hyperbolic space},
        date={2011},
   booktitle={Surveys in geometric analysis and relativity},
      series={Adv. Lect. Math. (ALM)},
      volume={20},
   publisher={Int. Press, Somerville, MA},
       pages={241\ndash 257},
}

\bib{guan2009jga}{article}{
      author={Guan, Bo},
      author={Spruck, Joel},
      author={Szapiel, Marek},
       title={Hypersurfaces of constant curvature in hyperbolic space. {I}},
        date={2009},
        ISSN={1050-6926},
     journal={J. Geom. Anal.},
      volume={19},
      number={4},
       pages={772\ndash 795},
         url={https://doi.org/10.1007/s12220-009-9086-7},
}

\bib{xiaoling2014jdg}{article}{
      author={Guan, Bo},
      author={Spruck, Joel},
      author={Xiao, Ling},
       title={Interior curvature estimates and the asymptotic plateau problem
  in hyperbolic space},
        date={2014},
        ISSN={0022-040X},
     journal={J. Differential Geom.},
      volume={96},
      number={2},
       pages={201\ndash 222},
         url={http://projecteuclid.org/euclid.jdg/1393424917},
}

\bib{lin1987}{article}{
      author={Hardt, Robert},
      author={Lin, Fang-Hua},
       title={Regularity at infinity for area-minimizing hypersurfaces in
  hyperbolic space},
        date={1987},
        ISSN={0020-9910},
     journal={Invent. Math.},
      volume={88},
      number={1},
       pages={217\ndash 224},
         url={https://doi.org/10.1007/BF01405098},
}

\bib{labourie1991}{article}{
      author={Labourie, Fran\c{c}ois},
       title={Probl\`eme de {M}inkowski et surfaces \`a courbure constante dans
  les vari\'{e}t\'{e}s hyperboliques},
        date={1991},
        ISSN={0037-9484},
     journal={Bull. Soc. Math. France},
      volume={119},
      number={3},
       pages={307\ndash 325},
         url={http://www.numdam.org/item?id=BSMF_1991__119_3_307_0},
}

\bib{yanyanlidegree}{article}{
      author={Li, Yan~Yan},
       title={Degree theory for second order nonlinear elliptic operators and
  its applications},
        date={1989},
        ISSN={0360-5302},
     journal={Comm. Partial Differential Equations},
      volume={14},
      number={11},
       pages={1541\ndash 1578},
         url={https://doi.org/10.1080/03605308908820666},
}

\bib{lieberman}{book}{
      author={Lieberman, Gary~M.},
       title={Second order parabolic differential equations},
   publisher={World Scientific Publishing Co., Inc., River Edge, NJ},
        date={1996},
        ISBN={981-02-2883-X},
         url={https://doi.org/10.1142/3302},
}

\bib{lin1989}{article}{
      author={Lin, Fang-Hua},
       title={On the {D}irichlet problem for minimal graphs in hyperbolic
  space},
        date={1989},
        ISSN={0020-9910},
     journal={Invent. Math.},
      volume={96},
      number={3},
       pages={593\ndash 612},
         url={https://doi.org/10.1007/BF01393698},
}

\bib{lopez2001}{article}{
      author={L\'{o}pez, Rafael},
       title={Graphs of constant mean curvature in hyperbolic space},
        date={2001},
        ISSN={0232-704X},
     journal={Ann. Global Anal. Geom.},
      volume={20},
      number={1},
       pages={59\ndash 75},
         url={https://doi.org/10.1023/A:1010676217144},
}

\bib{lopez1999}{article}{
      author={L\'{o}pez, Rafael},
      author={Montiel, Sebasti\'{a}n},
       title={Existence of constant mean curvature graphs in hyperbolic space},
        date={1999},
        ISSN={0944-2669},
     journal={Calc. Var. Partial Differential Equations},
      volume={8},
      number={2},
       pages={177\ndash 190},
         url={https://doi.org/10.1007/s005260050122},
}

\bib{lu2022}{article}{
      author={Lu, Siyuan},
       title={On the asymptotic plateau problem in hyperbolic space},
     journal={to appear in Proc. Amer. Math. Soc},
}

\bib{nellispruck1996}{incollection}{
      author={Nelli, Barbara},
      author={Spruck, Joel},
       title={On the existence and uniqueness of constant mean curvature
  hypersurfaces in hyperbolic space},
        date={1996},
   booktitle={Geometric analysis and the calculus of variations},
   publisher={Int. Press, Cambridge, MA},
       pages={253\ndash 266},
}

\bib{rosenberg1994}{article}{
      author={Rosenberg, Harold},
      author={Spruck, Joel},
       title={On the existence of convex hypersurfaces of constant {G}auss
  curvature in hyperbolic space},
        date={1994},
        ISSN={0022-040X},
     journal={J. Differential Geom.},
      volume={40},
      number={2},
       pages={379\ndash 409},
         url={http://projecteuclid.org/euclid.jdg/1214455540},
}

\bib{sui2021}{article}{
      author={Sui, Zhenan},
       title={Convex hypersurfaces with prescribed scalar curvature and
  asymptotic boundary in hyperbolic space},
        date={2021},
        ISSN={0944-2669},
     journal={Calc. Var. Partial Differential Equations},
      volume={60},
      number={1},
       pages={Paper No. 45, 29},
         url={https://doi.org/10.1007/s00526-020-01897-0},
}

\bib{sui2022}{article}{
      author={Sui, Zhenan},
      author={Sun, Wei},
       title={Smooth solutions to asymptotic {P}lateau type problem in
  hyperbolic space},
        date={2022},
        ISSN={1534-0392},
     journal={Commun. Pure Appl. Anal.},
      volume={21},
      number={10},
       pages={3353\ndash 3369},
         url={https://doi.org/10.3934/cpaa.2022103},
}

\bib{tonegawa}{article}{
      author={Tonegawa, Yoshihiro},
       title={Existence and regularity of constant mean curvature hypersurfaces
  in hyperbolic space},
        date={1996},
        ISSN={0025-5874},
     journal={Math. Z.},
      volume={221},
      number={4},
       pages={591\ndash 615},
         url={https://doi.org/10.1007/BF02622135},
}

\bib{wang2022}{article}{
      author={Wang, Bing},
       title={Curvature estimates for hypersurfaces of constant curvature in
  hyperbolic space},
     journal={to appear in Math. Res. Lett},
}

\end{biblist}
\end{bibdiv}

\end{document}